\documentclass[12pt,reqno]{amsart}
\usepackage{amssymb}
\usepackage{mathrsfs,mathtools}
\usepackage{palatino}
\usepackage{bbm}
\usepackage{pictex}
\usepackage{amscd}
\usepackage{tikz-cd}
\usepackage{latexsym}
\usepackage{amssymb}
\usepackage{enumitem}
\usepackage{eucal}
\usepackage{MnSymbol}
\usepackage{verbatim}
\usepackage{bbm}
\usepackage{hyperref}
\hypersetup{colorlinks,linkcolor=blue,urlcolor=cyan,citecolor=red}

\newcommand{\CC}{\mathbb{C}}

\newcommand{\NN}{\mathbb{N}}

\newcommand{\ZZ}{\mathbb{Z}}

\newcommand{\cP}{\mathcal{P}}

\newcommand{\Mat}{{\rm Mat}}
\newcommand{\Sym}{{\rm Sym}}
\newcommand{\Y}{{\rm \bf{Y}}}
\newcommand{\W}{{\rm \bf{W}}}
\newcommand{\E}{{\rm \bf{E}}}
\newcommand{\e}{{\rm \bf{e}}}
\newcommand{\U}{{\rm \bf{U}}}

\newcommand{\fg}{\mathfrak{g}}

\newcommand{\fn}{\mathfrak{n}}

\newcommand{\fsl}{\mathfrak{sl}}

\newcommand{\kk}{\mathbbm{k}}

\newcommand{\gl}{\mathfrak{gl}}

\newcommand{\mul}{\mathrm{mul}}
\newcommand{\unl}{\underline}

\numberwithin{equation}{section}
\newtheorem{Theorem}{Theorem}[section]
\newtheorem{Lemma}[Theorem]{Lemma}
\newtheorem{Corollary}[Theorem]{Corollary}
\newtheorem{Proposition}[Theorem]{Proposition}
\newtheorem{Remark}[Theorem]{Remark}
\newtheorem{Definition}[Theorem]{Definition}

\begin{document}
\title[Shuffle algebra realization for modular Yangians]{Shuffle algebra realization for modular Yangians}
\author[Hao Chang, Hongmei Hu, \lowercase{and}  Yue Hu  ]{Hao Chang, Hongmei Hu, \lowercase{and} Yue Hu  }
\address[H. Chang]{School of Mathematics and Statistics, Central China
Normal University, Wuhan 430079, China}
\email{chang@ccnu.edu.cn}
\address[H. Hu]{Department of Mathematics, Shanghai Maritime University, Shanghai 201306, China}
\email{hmhu@shmtu.edu.cn}
\address[Y. Hu]{School of Mathematical Sciences, Beijing University of Posts and Telecommunications, Beijing 102206, China}
\email{huyue@bupt.edu.cn}
\date{\today}
\makeatother
\begin{abstract}
We study the shuffle algebra realization of positive modular Yangians of classical type over an algebraically closed field of characteristic \(p>3\). We show that, unlike in characteristic zero, the natural shuffle homomorphism has a nontrivial kernel. Its image is characterized by a \(p\)-wheel condition, while its kernel is precisely the ideal generated by the \(p\)-th powers of the Lyndon root vectors. This identifies the corresponding quotient with the small Yangian arising from a \(\mathbb Z[\frac12]\)-integral form. As part of the construction, we establish a PBW basis for the integral form  and obtain the PBW theorem and \(p\)-center results for the Drinfeld presentation of modular Yangians.
\end{abstract}

\maketitle

\section{Introduction}

Shuffle algebras have long provided powerful tools in the study of quantum groups. A particularly fruitful recent development is the Feigin--Odesskii type shuffle algebra \cite{FO89,FO93,FO97,FO98} realization of quantum loop algebras and Yangians, which identifies their positive subalgebras in the Drinfeld presentation with an algebra of symmetric rational functions satisfying certain pole and wheel conditions \cite{Neg13,Tsy21,NT21,HT24,HT26}.  This shuffle algebra realization can be used to provide PBW bases and integral forms, and has also proven useful in the theory of integrable systems, via the construction of Lax matrices from shifted Yangians \cite{FPT22,FT22}, and in representation theory, via the development of category $\mathcal O$ and explicit realizations of simple modules for quantum loop algebras \cite{Neg26}.

The modular Yangian over an algebraically closed field $\kk$ of positive characteristic was studied by Brundan--Topley \cite{BT18} in type $A$, and by Chang--Hu \cite{CH25} in types $B,D$.  A fundamental feature of the modular theory is the existence of a large $p$-center.  In type $A$, Brundan--Topley described it in terms of Gaussian generators, and analogous results for the orthogonal Yangians were obtained in \cite{CH25}.  Thus the representation theory in positive characteristic differs substantially from the characteristic-zero case.  For the $A_1$-type Yangian $Y_2$, the finite-dimensional irreducible restricted representations were classified in \cite{CHT25}.

To advance the theory, it is natural to seek analogues of the powerful tools from modular Lie theory. There, the representation theory is approached via restricted enveloping algebras and the first Frobenius kernel, constructed using the Kostant $\ZZ$-form \cite{Hum77,Jan03}.  The quantum analogue was developed by Lusztig \cite{Lusztig90} and Chari--Pressley \cite{CP97} for quantum groups at roots of unity.  Our goal is to develop a corresponding framework for Yangians, with shuffle algebras providing the main tool.

The existing results for modular Yangians are mainly stated in the RTT presentation,
whereas shuffle algebras and Lyndon root vectors are naturally defined
in the Drinfeld presentation.  
For $p>3$, the Drinfeld presentation retains the same form as in characteristic zero, and the RTT--Drinfeld isomorphism follows from the Gauss decomposition using the characteristic-zero formulas of Jing--Liu--Molev \cite{JLM18}, cf. \cite[Theorem~3.1 and Remark~3.10]{CH25}.  
We use the RTT--Drinfeld
isomorphism to transfer the PBW theorem to the Drinfeld
presentation of the modular Yangian $Y_n(\kk)$. 
This gives a PBW basis
of $Y_n^>(\kk)$ in terms of the Lyndon root vectors $e_\beta(r)$. In types $A$, $B$, and $D$,
using the known Gaussian
$p$-center results of \cite{BT18,CH25}, we show
that
\[
 e_\beta(r)^p\in Z(Y_n(\kk))
 \qquad
 (\beta\in\Delta^+,\ r\in\NN).
\]
Thus the elements $e_\beta(r)^p$ generate a polynomial central
subalgebra of $Y_n^>(\kk)$.  The corresponding statement in type $C$
is formulated as a conjecture.

We next use modified specialization maps of \cite{HT24,HT26} to construct a
$\ZZ[\frac12]$-integral form $\Y_n^>$ of the positive Yangian
$Y_n^>(\CC)$, together with its PBW basis and shuffle algebra
realization.  From its reduction modulo $p$, we obtain the small
Yangian $\bar y_n^>(\kk)$.  We then study the natural homomorphism
\[
 \Psi_{\kk}:Y_n^>(\kk)\longrightarrow W^{(n)}(\kk)
\]
to the modular shuffle algebra.  Unlike in characteristic-zero
case, $\Psi_{\kk}$ is not an isomorphism.  We determine
its image using reduced specialization maps, and prove that
\[
 \ker(\Psi_{\kk})
 =
 \left\langle
 e_\beta(r)^p
 \mid
 \beta\in\Delta^+,\ r\in\NN
 \right\rangle.
\]
Consequently, $\Psi_{\kk}$ factors through the quotient
$Y_n^{>,[p]}$ and yields a commutative diagram of isomorphisms
\[
\begin{tikzcd}
Y^{>,[p]}_n
 \arrow[d, "\bar{\varphi}", "\cong"']
 \arrow[r, "\bar{\Psi}_{\kk}", "\cong"']
&
\tilde{W}^{(n)}(\kk)
\\
\bar{y}^{>}_n(\kk)
 \arrow[r, "\bar{\Psi}"', "\cong"]
&
\bar{W}^{(n)}(\kk)
 \arrow[u, "\phi", "\cong"']
\end{tikzcd}
\]
which identifies $Y_n^{>,[p]}$ with the small Yangian $\bar{y}^{>}_n(\kk)$ via their shuffle algebra realizations (cf. Theorem \ref{maintheorem}). 
In types $A$,
$B$, and $D$, this identifies the restricted positive Yangian with the small Yangian obtained from the
integral form, thereby providing a Yangian analogue of the modular Lie-theoretic
picture described above.

The paper is organized as follows:
\begin{itemize}[leftmargin=0.5cm]
    \item Section \ref{sec:preliminaries} collects the necessary preliminaries. We fix notation for classical simple Lie algebras and Lyndon words, recall the Drinfeld presentation of the positive Yangian and its shuffle algebra realization in characteristic zero, and review the Hall--Littlewood polynomials that will be used later.

\item Section \ref{sec:specialization} constructs a
$\ZZ[\frac12]$-integral form $\Y_n^>$, establishing its PBW basis
and shuffle algebra realization.  The main tool is a modification of
the specialization maps of \cite{HT24,HT26} that remains valid in positive characteristic.
   
   \item Section \ref{sec:modular-yangian-structure} uses the
RTT--Drinfeld isomorphism to transfer known results on modular
Yangians from the RTT realization to the Drinfeld presentation. 
In particular, we obtain the PBW theorem of $Y^{>}_n(\kk)$ in terms of Lyndon root vectors $e_{\beta}(r)$, and in types $A,B,D$ we show all the elements $e_\beta(r)^p$ belong to the center of $Y^{>}_n(\kk)$.

\item Section \ref{sec:modularshufflerealization} establishes the main
results of the paper.  We determine the kernel and image of the modular shuffle homomorphism. 
Using the integral form
constructed in Section \ref{sec:specialization},
we identify $Y_n^{>,[p]}$ with the small Yangian $\bar{y}^{>}_n(\kk)$ via their shuffle realization.
\end{itemize}

\section{Preliminaries}\label{sec:preliminaries}

\subsection{Classical Lie algebras}\label{subsec:classicalliealgebra}

In this subsection, we recall some basic results  
for the classical simple Lie algebras over $\mathbb{C}$.  In type $A_n$ ($n\geq 1$), let $N=n+1$. The Lie algebra $\mathfrak{sl}_N(\mathbb{C})$ is the subalgebra 
of traceless matrices in $\mathfrak{gl}_N(\mathbb{C})$, spanned by
\begin{equation}\label{eq:F_ij-type A}
   F_{ii}=E_{ii}-E_{i+1,i+1}\quad(1\leq i\leq n),\qquad
F_{ij}=E_{ij}\quad(1\leq i\neq j\leq N). 
\end{equation}
The Cartan subalgebra is $\mathfrak{h}=\bigoplus_{i=1}^{n}\mathbb{C}F_{ii}$. Let \(E\) be the \(N\)-dimensional space of diagonal matrices in
\(\mathfrak{gl}_N(\mathbb C)\), and let
\(\{\varepsilon_i\}_{i=1}^N\) be the standard basis of \(E^*\), so that
\[
 \varepsilon_i\bigl(\operatorname{diag}(t_1,\ldots,t_N)\bigr)=t_i.
\]
We equip \(E^*\) with the nondegenerate symmetric bilinear form $(\varepsilon_i,\varepsilon_j)=\delta_{ij}$. Via restriction to \(\mathfrak h\), we identify \(\mathfrak h^*\) with the \(n\)-dimensional subspace
\[
 \left\{
 \sum_{i=1}^N a_i\varepsilon_i
 \;\middle|\;
 \sum_{i=1}^N a_i=0
 \right\}
 =
 \bigoplus_{i=1}^n
 \mathbb C(\varepsilon_i-\varepsilon_{i+1})
 \subset E^*.
\]
The root space decomposition is 
\[\mathfrak{sl}_N(\mathbb C)
 =
 \mathfrak h
 \oplus
 \bigoplus_{1\leq i\neq j\leq N}\mathbb C F_{ij},\]
with root system 
$\Delta=\{\varepsilon_i-\varepsilon_j\mid i\neq j\}$.  Positive roots are 
$\Delta^{+}=\{\varepsilon_i-\varepsilon_j\mid i<j\}$, 
and simple roots are 
$\alpha_i=\varepsilon_i-\varepsilon_{i+1}$ $(1\leq i\leq n)$. A Chevalley basis $\{x_\beta\}_{\beta\in\Delta}\cup\{h_i\}_{i=1}^{n}$ is given by 
\begin{equation}\label{eq:chevalleybasis-A}
 h_i=F_{ii}=E_{ii}-E_{i+1,i+1}\ (1\leq i\leq n),\qquad   x_{\varepsilon_i-\varepsilon_j}=F_{ij}=E_{ij}\ (1\leq i\neq j\leq N) 
\end{equation}

In types $B_n$ ($n\geq 2$), $C_n$ ($n\geq 3$), $D_n$ ($n\geq 4$), 
set $i':=N+1-i$ with $N=2n+1$ for $B_n$, and $N=2n$ for $C_n$ and $D_n$. 
The Lie algebra $\mathfrak{g}\subset\mathfrak{gl}_N(\mathbb{C})$ is spanned by
\begin{equation}\label{eq:F_ij-BCD}
F_{ij}=
\begin{cases}
E_{ij}-E_{j'i'}, &  \text{for types}\ B_n, D_n,\\[2pt]
E_{ij}-\epsilon_i\epsilon_jE_{j'i'}, & \text{for type }\ C_n,
\end{cases}    
\end{equation}
in which $\epsilon_i=1$ for $ 1\leq i\leq n$, $\epsilon_i=-1$ for $n+1\leq i\leq 2n$. The commutation relations are
\begin{equation}\label{eq:commuformulaFij}
  [F_{ij}, F_{k\ell}] = \delta_{kj} F_{i\ell} - \delta_{\ell i} F_{kj}
- \delta_{k i'}\, \theta_i\theta_j\, F_{j'\ell} 
+ \delta_{\ell j'}\, \theta_{i'}\theta_{j'}\, F_{k i'},
\end{equation}
where $\theta_i\equiv 1$ for $B_n,D_n$, and $\theta_i=\epsilon_i$ for $C_n$. In all three types the Cartan subalgebra is 
$\mathfrak{h}=\bigoplus_{i=1}^{n}\mathbb{C}F_{ii}$ with 
$F_{ii}=E_{ii}-E_{i'i'}$.  Define $\varepsilon_i\in\mathfrak{h}^{*}$ by 
$\varepsilon_i(F_{jj})=\delta_{ij}$. Then 
$\mathfrak{g}=\mathfrak{h}\oplus\bigoplus_{\beta\in\Delta}\mathfrak{g}_\beta$ 
with the following positive roots and root spaces:
\begin{equation}\label{eq:positiveroots-BCD}
\begin{array}{c|c|c|c}
& B_n & C_n & D_n\\ \hline
\Delta^{+} & \varepsilon_i\pm\varepsilon_j\;(i<j),\;\varepsilon_i 
& \varepsilon_i\pm\varepsilon_j\;(i<j),\;2\varepsilon_i 
& \varepsilon_i\pm\varepsilon_j\;(i<j)\\[4pt]
\mathfrak{g}_{\varepsilon_i-\varepsilon_j} & \mathbb{C}F_{ij} & \mathbb{C}F_{ij} & \mathbb{C}F_{ij}\\[4pt]
\mathfrak{g}_{\varepsilon_i+\varepsilon_j} & \mathbb{C}F_{ij'} & \mathbb{C}F_{ij'} & \mathbb{C}F_{ij'}\\[4pt]
\mathfrak{g}_{\varepsilon_i} & \mathbb{C}F_{i,n+1} & - & -\\[4pt]
\mathfrak{g}_{2\varepsilon_i} & - & \mathbb{C}F_{ii'} & -
\end{array}
\end{equation}
The simple roots are
\begin{equation}\label{eq:simpleroots-BCD}
\begin{aligned}
B_n&:\; \alpha_i=\varepsilon_i-\varepsilon_{i+1}\;(i<n),\;\alpha_n=\varepsilon_n;\\
C_n&:\; \alpha_i=\varepsilon_i-\varepsilon_{i+1}\;(i<n),\;\alpha_n=2\varepsilon_n;\\
D_n&:\; \alpha_i=\varepsilon_i-\varepsilon_{i+1}\;(i<n),\;\alpha_n=\varepsilon_{n-1}+\varepsilon_n.
\end{aligned}
\end{equation}
We equip $\mathfrak{h}^{*}$ with a nondegenerate symmetric $\mathbb{C}$-bilinear form 
$(\cdot,\cdot)$ such that $\{\varepsilon_i\}$ is an orthonormal basis:
$(\varepsilon_i,\varepsilon_j)=\delta_{ij}$. In terms of the generators $F_{ij}$, a Chevalley basis $\{x_\beta\}_{\beta\in\Delta}\cup\{h_i\}_{i=1}^{n}$ is given by (we list only the basis element corresponding to the Cartan part and positive roots):
\begin{equation}\label{eq:chevalleybasis-BCD}
\begin{array}{c|c|c|c}
& B_n & C_n & D_n\\ \hline
x_{\varepsilon_i-\varepsilon_j} & F_{ij} & F_{ij} & F_{ij}\\
x_{\varepsilon_i+\varepsilon_j} & F_{ij'} & F_{ij'} & F_{ij'}\\
x_{\varepsilon_i} & \sqrt{2}\,F_{i,n+1} & - & -\\
x_{2\varepsilon_i} & - & \frac12 F_{ii'} & -\\[4pt] \hline
h_i\;(i<n) & F_{ii}-F_{i+1,i+1} & F_{ii}-F_{i+1,i+1} & F_{ii}-F_{i+1,i+1}\\
h_n & 2F_{nn} & F_{nn} & F_{n-1,n-1}+F_{nn}
\end{array}
\end{equation}

For any classical simple Lie algebra $\fg$, let $\{\alpha_i\}_{i=1}^{n}$ be the 
simple roots as above. Set $d_i:=\frac{(\alpha_i,\alpha_i)}{2}$. 
Thus in type $A_n$ and type $D_n$, $d_i=1$ for all $1\leq i\leq n$; in type $B_n$, $d_i=1$ for $i<n$ 
and $d_n=\frac12$; in type $C_n$, $d_i=1$ for $i<n$ and $d_n=2$. 
The Cartan matrix $C=(c_{ij})_{i,j=1}^{n}$ is defined by
\begin{equation}\label{eq:cartan-matrix}
c_{ij}:=\frac{2(\alpha_i,\alpha_j)}{(\alpha_i,\alpha_i)}
  ,
\end{equation}
and satisfies $d_i c_{ij}=(\alpha_i,\alpha_j)=d_j c_{ji}$. Conversely, the Lie algebra $\mathfrak{g}$ can be reconstructed from the 
Cartan matrix $C$ via the Chevalley--Serre presentation.  Let $e_i:=x_{\alpha_i}$, $f_i:=x_{-\alpha_i}$ be the root vectors 
corresponding to the simple roots, then $\{e_i,f_i,h_i\}_{i=1}^{n}$ \ (cf. (\ref{eq:chevalleybasis-A}, \ref{eq:chevalleybasis-BCD})) generate $\mathfrak{g}$ and satisfy 
the following defining relations:
\begin{equation}\label{eq:Serre-relations}
\begin{aligned}
& [h_i,h_j]=0,\quad
[h_i,e_j]=c_{ij}e_j,\quad
[h_i,f_j]=-c_{ij}f_j,\quad
[e_i,f_j]=\delta_{ij}h_i,\\[4pt]
& (\mathrm{ad}\  e_i)^{1-c_{ij}}(e_j)=0,\quad
(\mathrm{ad}\ f_i)^{1-c_{ij}}(f_j)=0\qquad (i\neq j).
\end{aligned}
\end{equation}

 Each $\beta\in\Delta^{+}$ can be uniquely expressed as a sum
of simple roots: 
\begin{equation}\label{eq:coefficientpositive}
 \beta=\sum_{1\leq i \leq n}\nu_{\beta,i}\alpha_{i},\qquad  \text{with} \ \nu_{\beta,i}\in\NN.   
\end{equation}
 We shall refer to
$\nu_{\beta,i}$ as the \emph{coefficient of $\alpha_{i}$ in $\beta$}, and  use the following notation: $i\in\beta  \Longleftrightarrow \nu_{\beta,i}\neq 0$. The height of a root $\beta\in \Delta^+$ is 
 $ |\beta|:=\sum_{1\leq i\leq n} \nu_{\beta,i}$. For later use, 
it is convenient to parametrize the positive roots via standard Lyndon words, as in \cite{HT24,HT26} (cf. \cite{Lec04}):

\begin{equation}\label{eq:lyndon-notation}
\begin{array}{c|c|c}
\text{Type} & \text{Positive root} & \text{Standard Lyndon words}\\[4pt] \hline
\rule{0pt}{16pt} 
A_n & \varepsilon_i-\varepsilon_{j+1}\ (1\leq i\leq j\leq n) & [i,j]:=[i,\ldots,j]\\[4pt]
\hline
\rule{0pt}{16pt}
 & \varepsilon_i-\varepsilon_{j+1}\ (1\leq i\leq j\leq n-1) & [i,j]:=[i,\ldots,j]\\[4pt]
B_n & \varepsilon_i\ (1\leq i\leq n) & [i,n]:=[i,\ldots,n]\\[4pt]
 & \varepsilon_i+\varepsilon_j\ (1\leq i<j\leq n) & [i,n,j]:=[i,\ldots,n,n,n-1,\ldots,j]\\[8pt]
\hline
\rule{0pt}{16pt}
 & \varepsilon_i-\varepsilon_{j+1}\ (1\leq i\leq j\leq n-1) & [i,j]:=[i,\ldots,j]\\[4pt]
 & \varepsilon_i+\varepsilon_{n}\ (1\leq i\leq n-1) & [i,n]:=[i,\ldots,n]\\[4pt]
C_n & \varepsilon_i+\varepsilon_j\ (1\leq i<j< n) & [i,n,j]:=[i,\ldots,n-1,n,n-1,\ldots,j]\\[4pt]
 & 2\varepsilon_n & [n]\\[4pt]
 & 2\varepsilon_i\ (1\leq i< n) & [i,n,i]:=[i,\ldots,n-1,i,\ldots,n-1,n]\\[8pt]
\hline
\rule{0pt}{16pt}
 & \varepsilon_i-\varepsilon_{j+1}\ (1\leq i\leq j<n) & [i,j]:=[i,\ldots,j]\\[4pt]
D_n & \varepsilon_i+\varepsilon_n\ (1\leq i\leq n-1) & [i,n]:=[i,\ldots,n-2,n]\\[4pt]
 & \varepsilon_i+\varepsilon_j\ (1\leq i<j<n) & [i,n,j]:=[i,\ldots,n-2,n,n-1,\ldots,j]
\end{array}
\end{equation}
The lexicographical order on the standard Lyndon words gives rise to a convex order $<$ on $\Delta^{+}$, cf. \cite[Proposition 26]{Lec04}. In what follows, we fix this
specific order on $\Delta^{+}$.

\subsection{Drinfeld presentation of the Yangian $Y^{>}_n$}\label{subsec:Drinfeldpre}

Let $\mathfrak{g}$ be a classical simple Lie algebra of rank $n$ over $\mathbb{C}$. Following Drinfeld \cite{Drin88} (cf. \cite{Mol07}), we define 
the {\it positive subalgebra} of the Yangian associated to $\mathfrak{g}$, 
denoted by $Y^{>}_n$, to be the associative $\mathbb{C}$-algebra generated by 
$\{e_{i,r}\mid 1\leq i\leq n,\ r\in\mathbb{N}\}$ subject to the following defining relations:
\begin{align}
&[e_{i,r+1},e_{j,s}]-[e_{i,r},e_{j,s+1}]=\tfrac{d_i c_{ij}}{2}(e_{i,r}e_{j,s}+e_{j,s}e_{i,r})
\qquad \forall\ 1\leq i,j\leq n,\ r,s\in\mathbb{N},\label{eq:Drinfeld-rel1}\\[4pt]
&\mathop{\mathrm{Sym}}_{s_1,\dots,s_{1-c_{ij}}}[e_{i,s_1},[e_{i,s_2},\cdots,[e_{i,s_{1-c_{ij}}},e_{j,r}]\cdots]]=0
\qquad \forall\ 1\leq i\neq j\leq n,\ s_1,\dots,s_{1-c_{ij}},r\in \mathbb{N}.\label{eq:Drinfeld-rel2}
\end{align}

Recall the Lyndon presentation of the positive roots, cf. \eqref{eq:lyndon-notation}. The fixed total order on $\Delta^{+}$  induces a total order on $\Delta^{+}\times\mathbb{N}$ as follows:
\begin{equation}\label{eq:order-Delta-N}
(\beta,r)\leq (\beta',r') \ \text{if}\ \beta<\beta' \ \text{or}\ \beta=\beta',\ r\geq r'.
\end{equation}
Now we introduce the root vectors in $Y^{>}_n$. For  $\beta=[i_1,\dots,i_\ell]\in \Delta^{+}$ and $r\in\NN$, we define 
$e_{\beta}(r)\in Y^{>}_n$ as follows 
(cf. \cite{HT24}--\cite{HT26}):

\begin{itemize}[leftmargin=0.7cm]
\item If $\mathfrak{g}$ is of type $B_n$ and $\beta=[i,n]$ (with $1\leq i\leq n$), we set \footnote{Here the choice of root vectors differs from that of
\cite{HT24}--\cite{HT26}, since it
is convenient for the study of $p$-center of modular Yangian, cf. Subsection \ref{subsec:modular-pcenter}. }
\begin{equation}\label{eq:B-root-vector-special}
e_{\beta}(r):=[\cdots[[e_{i,0},e_{i+1,0}],e_{i+2,0}],\cdots,e_{n,r}],
\end{equation}
\item If $\mathfrak{g}$ is of type $C_n$ and $\beta=[i,n,i]$ (with $1\leq i<n$), we set
\begin{equation}\label{eq:C-root-vector-special}
e_{\beta}(r):=
[[\cdots[e_{i,0},e_{i+1,0}],\cdots,e_{n-1,0}],\,
[[\cdots[e_{i,0},e_{i+1,0}],\cdots,e_{n-1,0}],e_{n,r}]].
\end{equation}
\item In all other cases, we set
\begin{equation}\label{eq:root-vector-standard}
e_{\beta}(r):=[\cdots[[e_{i_1,r},e_{i_2,0}],e_{i_3,0}],\cdots,e_{i_\ell,0}].
\end{equation}
\end{itemize}

Let $\mathcal{H}$ denote the set of all functions 
$h\colon \Delta^{+}\times\mathbb{N}\rightarrow \mathbb{N}$ with finite support. 
For any $h\in\mathcal{H}$, we consider the ordered monomial
\begin{equation}\label{eq:pbwd-yangian}
E_{h}\, =\prod_{(\beta,r)\in\Delta^{+}\times\mathbb{N}}^{\rightarrow} e_{\beta}(r)^{\,h(\beta,r)},
\end{equation}
where the arrow $\rightarrow$ over the product sign refers to the total order 
\eqref{eq:order-Delta-N}. Then, following \cite{Lev93} (cf. \cite[Theorem B.3]{FT19}), we have the PBW theorem for $Y^>_n$:

\begin{Theorem}\label{thm:yangian-basis}
The elements $\{E_{h}\}_{h\in\mathcal{H}}$ defined by \eqref{eq:pbwd-yangian} form 
a $\mathbb{C}$-basis of $Y^{>}_n$.
\end{Theorem}

\subsection{Shuffle algebra realization of $Y^>_n$}\label{subses:shuffleyangian}
Following \cite[Section 6]{Tsy21}\cite[Section 5]{HT24}\cite[Section 5]{HT26},  let us now recall the {\it Feigin-Odesskii shuffle algebra} and the shuffle algebra realization of $Y^>_n$.
Let $\Sigma_k$ denote the symmetric group in $k$ elements, and set $\Sigma_{(k_1,\dots,k_{n})}:=\Sigma_{k_1}\times\cdots\times \Sigma_{k_{n}}$ for $k_1,\dots,k_{n}\in\NN$.
Consider the following $\NN^{n}$-graded $\CC$-vector space
\[ 
  W^{(n)}=\bigoplus_{\underline{k}=(k_1,\dots,k_n)\in \mathbb{N}^{n}}W^{(n)}_{\underline{k}},
\]
where $W^{(n)}_{\underline{k}}$ consists of rational functions $F$ in the variables 
$\{x_{i,r}\}_{1\leq i\leq n}^{1\leq r\leq k_{i}}$ such that:
\begin{itemize}[leftmargin=0.7cm]

\item
$F$ is $\Sigma_{\unl{k}}$-symmetric, that is, symmetric in $\{x_{i,r}\}_{r=1}^{k_{i}}$ for each $1\leq i\leq n$,

\medskip
\item
(\emph{pole conditions}) $F$ has the form
\begin{equation}
\label{polecondition}
  F=\frac{f(\{x_{i,r}\}_{1\leq i\leq n}^{1\leq r\leq k_{i}})}
         {\prod_{i<j}^{c_{ij}\neq 0}\prod_{1\leq r\leq k_{i}}^{1\leq s\leq k_{j}}(x_{i,r}-x_{j,s})},
\end{equation}
where $f\in \CC[\{x_{i,r}\}_{1\leq i\leq n}^{1\leq r\leq k_{i}}]^{\Sigma_{\underline{k}}}$,

\medskip
\item
(\emph{wheel conditions}) Let $f$ be the numerator of $F\in W^{(n)}_{\unl{k}}$ from \eqref{polecondition}, then
\begin{equation}
\label{wheel-Y}
  f(\{x_{i,r}\}_{1\leq i\leq n}^{1\leq r\leq k_{i}})=0 \  \text{once}\
  x_{i,s_{1}}=x_{i,s_{2}}+d_{i}=\cdots=x_{i,s_{1-c_{ij}}}-d_{i}c_{ij}=x_{j,r}-\frac{d_{i}c_{ij}}{2}
\end{equation}
for any $ 1\leq i\neq j\leq n$ such that $c_{ij}\neq 0$, pairwise distinct $1\leq s_{1},\dots,s_{1-c_{ij}}\leq k_{i}$, and $1\leq r\leq k_{j}$.
\end{itemize}

\begin{Definition}\label{def:shuffleelements}
For $\unl{k}=(k_1,\dots,k_{n}), \unl{\ell}=(\ell_1,\dots,\ell_{n})\in\NN^{n}$, the set $\text{sh}_{\unl{k},\unl{\ell}}$ of $(\unl{k},\unl{\ell})$-shuffles  is the set of elements $\sigma_1\times \cdots \times \sigma_{n}\in\Sigma_{\unl{k}+\unl{\ell}}$ which satisfy:
\begin{equation*}
 \sigma_i(1)<\cdots <\sigma_i(k_i),\quad \sigma_i(k_i+1)<\cdots <\sigma_i(k_i+\ell_i),\qquad \text{for any}\ 1\leq i\leq n.   
\end{equation*}   
 \end{Definition}

We fix an $n\times n$ matrix of rational functions $(\zeta_{ij}(z))_{1\leq i,j\leq n}\in\Mat_{n\times n}(\CC(z))$ via
\begin{equation}\label{eq:zetafactor}
    \zeta_{ij}(z):=1+\frac{d_ic_{ij}}{2z}.
    \end{equation}
    Given $F\in W^{(n)}_{\underline{k}}$ and $G\in W^{(n)}_{\underline{\ell}}$,
we define the {\it shuffle product}  $\star$ on $W^{(n)}$ via
\begin{equation}\label{shuffle product-1}
\begin{aligned}
W^{(n)}_{\underline{k}+\underline{\ell}}\ni & (F\star G)(x_{1,1},\dots, x_{1,k_1+\ell_1};\dots;x_{n,1},\dots,x_{n,k_{n}+\ell_{n}}):=\\
   & \Sym_{\text{sh}_{\unl{k},\unl{\ell}}}\Big(
   F\big(\{x_{i,r}\}_{1\leq i\leq  n}^{1\leq r\leq k_i}\big) G\big(\{x_{i',r'}\}_{1\leq i'\leq n}^{k_{i'}< r'\leq k_{i'}+\ell_{i'}}\big)\cdot\prod\limits_{1\leq i,i'\leq  n}\prod\limits_{r\leq k_i}^{r'>k_{i'}}\zeta_{ii'}(x_{i,r}-x_{i',r'})\Big).
\end{aligned}
\end{equation}
where $\zeta$-factor $\zeta_{ii'}(z)$ is as \eqref{eq:zetafactor}, the symmetrization associated to $\text{sh}_{\unl{k},\unl{\ell}}$ is as
\begin{align*}
\Sym_{\text{sh}_{\unl{k},\unl{\ell}}}(f)\big(\{x_{i,1},\dots,x_{i,m_i}\}_{1\leq i\leq n}\big) :=\sum\limits_{(\sigma_1,\dots,\sigma_{n})\in \text{sh}_{\unl{k},\unl{\ell}}}f\big(\{x_{i,\sigma_i(1)},\dots,x_{i,\sigma_i(m_i)}\}_{1\leq i\leq n}\big).
\end{align*}

\begin{Remark}
Following \cite[Remarks 3--4]{Enr00}, the product \eqref{shuffle product-1} is the same as the product used in \cite[Sect. 6.2]{Tsy21}:
\begin{equation*}
\begin{aligned}
W^{(n)}_{\underline{k}+\underline{\ell}}\ni & (F\star G)(x_{1,1},\dots, x_{1,k_1+\ell_1};\dots;x_{n,1},\dots,x_{n,k_{n}+\ell_{n}}):=\frac{1}{\underline{k}!\cdot\underline{\ell}!}\\
   &\times \Sym_{\Sigma_{\underline{k}+\underline{\ell}}}\Big(
   F\big(\{x_{i,r}\}_{1\leq i\leq  n}^{1\leq r\leq k_i}\big) G\big(\{x_{i',r'}\}_{1\leq i'\leq  n}^{k_{i'}< r'\leq k_{i'}+\ell_{i'}}\big)\cdot\prod\limits_{1\leq i,i'\leq n}\prod\limits_{r\leq k_i}^{r'>k_{i'}}\zeta_{ii'}(x_{i,r}-x_{i',r'})\Big),
\end{aligned}
\end{equation*}
where $\underline{k}!:=\prod_{1\leq i\leq n}k_i!$, while the {\it symmetrization} of $f\in\CC\big(\{x_{i,1},\dots,x_{i,m_i}\}_{1\leq i\leq n}\big)$ is defined via
\begin{align*}
  \Sym_{\Sigma_{\underline{m}}}(f)\big(\{x_{i,1},\dots,x_{i,m_i}\}_{1\leq i\leq n}\big) :=\sum\limits_{(\sigma_1,\dots,\sigma_{n})\in\Sigma_{\underline{m}}}f\big(\{x_{i,\sigma_i(1)},\dots,x_{i,\sigma_i(m_i)}\}_{1\leq i\leq n}\big).
\end{align*}
 The advantage of \eqref{shuffle product-1} is that it also applies to modular case, cf. Section \ref{sec:modularshufflerealization}.
\end{Remark}

It is straightforward to check that  $W^{(n)}$
is closed under $\star$. The resulting associative $\CC$-algebra $W^{(n)}$ is called the {\it rational shuffle algebra}. It is related to $Y_n^>$ via the following result (cf. \cite[Theorem 6.20]{Tsy21}\cite[Theorem 5.14]{HT24}\cite[Theorem 5.11]{HT26}):
\begin{Theorem}\label{thm:shufflerealizationyangian}
The assignment $e_{i,r}\mapsto x_{i,1}^r~(1\leq i\leq n,r\in\NN)$ gives rise to a $\CC$-algebra isomorphism 
\begin{equation}\label{eq:shufflerealizationyangian}
\Psi\colon Y_n^> \stackrel{\sim}{\longrightarrow}W^{(n)}.     
\end{equation}  
\end{Theorem}

The root vectors $e_{\beta}(r)$ of \eqref{eq:B-root-vector-special}--\eqref{eq:root-vector-standard} admit remarkably simple images under the  isomorphism $\Psi$, which we summarize for all classical
types, cf. \cite{Tsy21}\cite{HT24}\cite{HT26}:
\begin{itemize}[leftmargin=0.7cm]

\item When $\beta=[i,j]$, except for the short root $[i,n]$ in type $B_n$, we have:
\begin{align}\label{eq:psirv-1}
\Psi(e_{[i,j]}(r)) \doteq \frac{x_{i,1}^r}{\mathrm{denom}_{[i,j]}}.
\end{align}
For $\beta=[i,n]$ in type $B_n$, we have 
\begin{align}\label{eq:psirv-B-short}
\Psi(e_{[i,n]}(r)) \doteq \frac{x_{n,1}^r}{\mathrm{denom}_{[i,n]}}.
\end{align}

\item When $\beta=[i,n,j]$:

For type $B_n$, we have
\begin{align}
\Psi(e_{[i,n,j]}(r)) \doteq \frac{x_{i,1}^r}{\mathrm{denom}_{[i,n,j]}}
\prod_{\ell=j}^{n-1} (x_{\ell,1}-x_{\ell,2}+1)(x_{\ell,2}-x_{\ell,1}+1).
\end{align}

For type $C_n$,  we have
\begin{align}
\Psi(e_{[i,n,j]}(r)) \doteq \frac{x_{i,1}^r}{\mathrm{denom}_{[i,n,j]}}
(2x_{j-1,1}-x_{j,1}-x_{j,2})
\prod_{\ell=j}^{n-2} Q(x_{\ell,1},x_{\ell,2},x_{\ell+1,1},x_{\ell+1,2}),
\end{align}

For type $D_n$, we have
\begin{align}
\Psi(e_{[i,n,j]}(r)) \doteq \frac{x_{i,1}^r}{\mathrm{denom}_{[i,n,j]}}
\prod_{\ell=j}^{n-2} (x_{\ell,1}-x_{\ell,2}+1)(x_{\ell,2}-x_{\ell,1}+1).
\end{align}

\item When $\beta=[i,n,i]$ (only in type $C_n$), we have
\begin{align}\label{eq:psirv-2}
\Psi(e_{[i,n,i]}(r)) \doteq \frac{x_{n,1}^r}{\mathrm{denom}_{[i,n,i]}}
\prod_{\ell=i}^{n-2} Q(x_{\ell,1},x_{\ell,2},x_{\ell+1,1},x_{\ell+1,2}).
\end{align}

\end{itemize}  
Here, we use $\mathrm{denom}_{\beta}$ to denote the denominator of any $F=\Psi(e_{\beta}(r))$  in \eqref{polecondition}, $Q$ to denote the polynomial:
\begin{equation}\label{eq:Qpolyonimal}
Q(x_1,x_2,y_1,y_2)=4(x_1x_2+y_1y_2)-2(x_1+x_2)(y_1+y_2)+1,  
\end{equation}
and $\doteq$ to denote an equality up to $\ZZ[\frac12]^{\times}$, that is
\begin{equation}\label{eq:doteqmean}
 A\doteq B \quad \text{if} \quad  A=c\cdot B \quad \text{for some} \ c=\pm 2^{t}\ (t\in\ZZ).
\end{equation}

\subsection{Hall-Littlewood polynomials}\label{subs:zbasis}

In type $A_1$,  the rational shuffle algebra $W^{(1)}$ consists  of symmetric polynomials (there are no pole conditions and wheel conditions):
\[ W^{(1)}=\bigoplus_{k\in\NN} W^{(1)}_k,\quad \text{in which}\quad  W^{(1)}_k=\CC[x_1,\dots,x_k]^{\Sigma_k},\]
where we denote the variable $x_{1,r}$ in $W^{(1)}$ simply by $x_r$. Consider the following more general shuffle product with a parameter $\hbar\in \ZZ[\frac12]$: 
\begin{equation}\label{eq:shuffleproducta1hbar} 
(F\star G)(x_{1},\dots, x_{k+\ell})=\Sym_{\text{sh}_{k,\ell}}\left(
   F\left(\{x_{r}\}_{1\leq r\leq k}\right) G\left(\{x_{r'}\}_{k< r'\leq k+\ell}\right)\cdot\prod\limits_{r\leq k}^{r'>k}\frac{x_{r}-x_{r'}+\hbar}{x_{r}-x_{r'}}\right).
    \end{equation}
The resulting associative algebra will be denoted by $W^{(1)}_{\hbar}$. The following result is from \cite[Lemma 6.22]{Tsy21}:
\begin{Lemma}\label{k times of r-powerof x1}
For any $k\geq 1$ and $r\in\NN$, the $k$-th power of $x_1^r\in W_{\hbar}^{(1)}$ equals
\begin{align}\label{equ:k times of r-powerof x1}
\underbrace{x_1^r\star\cdots\star x_1^r}_{k~{\rm times}}=k!\cdot(x_1\cdots x_k)^r.  
\end{align}
\end{Lemma}

Let $\cP_k$ be the set of partitions of length $k$, where   $\lambda\in\cP_k$ is denoted by 
\[\lambda=\left\{\lambda_1\geq \lambda_2\geq\cdots\geq \lambda_k\right\},\quad \text{in which}\quad   \lambda_1,\lambda_2\dots,\lambda_k\in\NN,\]
 and $\cP=\bigcup_{k\in\NN} \cP_k$ be the set of all partitions. For $\lambda\in\cP$, we denote the number of parts $\lambda_j$ which are equal to $i$ by $m_i(\lambda)$,  $\mul(\lambda)=\prod^{\infty}_{i=0} m_i(\lambda)!$,  and the size of $\lambda$ by $|\lambda|\coloneqq \lambda_1+\lambda_2+\cdots+\lambda_k$. For $\lambda\in\cP_k$,  the following is a variant of {\em Hall-Littlewood polynomial} (cf. \cite[chapter III]{Mac95}):
\begin{equation}\label{eq:halllittlewoodpolynomial}
\begin{aligned}
P_\lambda&\coloneqq \frac{1}{\mul(\lambda)}\cdot  x^{\lambda_1}_1\star x^{\lambda_2}_1\star\cdots\star x^{\lambda_k}_1\\
&=\frac{1}{\mul(\lambda)}\cdot \Sym_{\Sigma_k}\left(x^{\lambda_1}_1x^{\lambda_2}_2\cdots x^{\lambda_k}_k\prod_{1\leq i<j\leq k}\frac{x_i-x_j+\hbar}{x_i-x_j}\right).
\end{aligned}
\end{equation} 

Let  $\W^{(1)}=\bigoplus_{k\in\NN}\ZZ[\frac12][x_1,\dots,x_k]^{\Sigma_k}$ be the ring of symmetric polynomials over $\ZZ[\frac12]$, then the set of  monomial symmetric polynomials   
\[m_\lambda=\frac{1}{\mul(\lambda)}\cdot \Sym_{\Sigma_k}\left(x^{\lambda_1}_1x^{\lambda_2}_2\cdots x^{\lambda_k}_k\right)\quad  (\lambda\in\cP)\]
form a basis of $\W^{(1)}$ over $\ZZ[\frac12]$. Since $\hbar\in \ZZ[\frac12]$, it follows from Lemma \ref{k times of r-powerof x1} and \eqref{eq:halllittlewoodpolynomial} that $P_\lambda\in \ZZ[\frac12][x_1,\dots,x_k]^{\Sigma_k}$, and 
\begin{equation}\label{eq:plambdamlambda}
P_\lambda=m_\lambda+\sum_{|\mu|<|\lambda|} a_{\lambda\mu} m_{\mu},\quad \text{in which}\quad  a_{\lambda\mu}\in\ZZ[\tfrac{1}{2}].  
\end{equation}
Thus $\{P_\lambda\}_{\lambda\in\cP}$ form also a $\ZZ[\frac12]$-basis of $\W^{(1)}$.


\section{Specialization map approach}\label{sec:specialization}

\subsection{Specialization maps}

Identifying each simple root $\alpha_{i}\ (1\leq i\leq n)$ with a basis element $\mathbf{1}_{i}\in\mathbb{N}^{n}$ (having 
the $i$-th coordinate equal to $1$ and the rest equal to $0$), we can view $\mathbb{N}^{n}$ as the positive cone of 
the root lattice of $\fg$. For any $\underline{k}\in\mathbb{N}^{n}$, let $\text{KP}(\underline{k})$ be the set of 
{\em Kostant partitions}, i.e.\ unordered vector partitions of $\underline{k}$ into a sum of positive roots. 
Explicitly, a Kostant partition of $\underline{k}$ is the same as a tuple 
$\unl{d}=\{d_{\beta}\}_{\beta\in\Delta^{+}}\in \NN^{\Delta^{+}}$ satisfying 
  $\sum_{i\in I} k_i\alpha_i \, = \sum_{\beta\in\Delta^{+}}d_{\beta}\beta$. The lexicographical order on $\Delta^{+}$ induces a total order on 
$\text{KP}(\underline{k})$:
\begin{equation}
\label{eq:KP-order}
  \{d'_\beta\}_{\beta\in \Delta^+}<\{d_\beta\}_{\beta\in \Delta^+}\Longleftrightarrow \exists\  \gamma\in \Delta^+\
  \mathrm{s.t.}\ d'_\gamma<d_\gamma\ \mathrm{and}\
  d'_\beta=d_\beta\ \mathrm{for\ all}\ \beta<\gamma.
\end{equation}

We now define the specialization maps for $W^{(n)}$. These maps are modifications of those in \cite{HT24,HT26} and are designed to remain valid in positive characteristic (cf. Section \ref{sec:modularshufflerealization}).  For any $F\in W^{(n)}_{\unl{k}}$ and 
$\underline{d}\in\text{KP}(\underline{k})$, we split the variables $\{x_{i,\ell}\}_{1\leq i\leq  n}^{1\leq \ell\leq k_{i}}$
into the disjoint union of $\sum_{\beta\in\Delta^{+}}d_{\beta}$ groups
\begin{equation}
\label{splitvariable}
  \bigsqcup^{1\leq s\leq d_{\beta}}_{\beta\in\Delta^{+}}
  \Big\{x^{(\beta,s)}_{i,t} \, \Big| \, 1\leq i\leq n, 1\leq t\leq \nu_{\beta,i}\Big\} \,,
\end{equation}
where the integer $\nu_{\beta,i}$ is the coefficient of $\alpha_i$ in $\beta$, cf. \eqref{eq:coefficientpositive}. For any $F\in W^{(n)}_{\unl{k}}$ 
and $\unl{d}\in\text{KP}(\unl{k})$, let $f$ be the numerator of $F$ from  \eqref{polecondition}. For each type, the specialization map 
$\phi_{\underline{d}}(F)$ is defined by successive specializations $\phi_{\beta,s}$ of the variables $x^{(\beta,s)}_{*,*}$ 
in $f$ for each $\beta\in\Delta^+$ and $1\leq s\leq d_{\beta}$ as follows:
\begin{itemize}[leftmargin=0.7cm]

\item 

$A_n$-type: for $\beta=[i,j]$ with $1\leq i\leq j\leq n$, $1\leq s\leq d_{\beta}$, we define $\phi_{\beta,s}(F)$ by specializing:
\begin{equation}\label{spe-A}
 x^{(\beta,s)}_{\ell,1}\mapsto w_{\beta,s}-\tfrac{\ell}{2},\qquad  (\ell\in\beta).
\end{equation}

\item 
$B_n$-type: for $\beta=[i,j]$ with $1\leq i\leq j\leq n$, $1\leq s\leq d_{\beta}$, we define $\phi_{\beta,s}(F)$ by specializing:
\begin{equation}\label{spe-B-1}
 x^{(\beta,s)}_{\ell,1}\mapsto w_{\beta,s}-\tfrac{\ell}{2},\qquad  ( \ell\in\beta).
\end{equation}
For $\beta=[i,n,j]$ with $1\leq i<j\leq n$, $1\leq s\leq d_{\beta}$, we first define $\phi^{(1)}_{\beta,s}(F)$ by specializing:
\begin{equation*}
 \left\{
    \begin{aligned}
        x^{(\beta,s)}_{\ell,1}&\mapsto w_{\beta,s}-\tfrac{\ell}{2}, \\
       x^{(\beta,s)}_{n,2}&\mapsto w_{\beta,s}-\tfrac{n-1}{2}, \\
        x^{(\beta,s)}_{\ell\neq n,2}&\mapsto w'_{\beta,s}-\tfrac{2n-\ell-1}{2}.
    \end{aligned}
    \right.
\end{equation*}
According to wheel conditions \eqref{wheel-Y}, $\phi^{(1)}_{\beta,s}(F)$ is divisible by 
\begin{equation*}
  B_{\beta}=\prod^{n-1}_{\ell=j}\big\{(w_{\beta,s}-w'_{\beta,s}+n-\ell-\tfrac{3}{2} )(w_{\beta,s}-w'_{\beta,s}+n-\ell+\tfrac{1}{2} )\big\}.
\end{equation*}
Then the overall specialization $\phi_{\beta,s}(F)$ is defined by
\begin{equation}
\label{spe-B-2}
  \phi_{\beta,s}(F)\coloneqq \phi^{(2)}_{\beta,s}\left(\phi^{(1)}_{\beta,s}(F)\right) = \left. \frac{\phi^{(1)}_{\beta,s}(F)}{B_{\beta}}\right|_{w'_{\beta,s}\mapsto w_{\beta,s}}.
\end{equation}
\item 
$C_n$-type: for $\beta=[i,j]$ with $1\leq i\leq j\leq n$, $1\leq s\leq d_{\beta}$, we define $\phi_{\beta,s}(F)$ by specializing:
\begin{equation}\label{spe-C-1}
 \left\{
    \begin{aligned}
        x^{(\beta,s)}_{\ell\neq n,1}&\mapsto w_{\beta,s}-\tfrac{\ell}{2}, \\
       x^{(\beta,s)}_{n,1}&\mapsto w_{\beta,s}-\tfrac{n+1}{2}. 
    \end{aligned}
    \right.
\end{equation}
For $\beta= [i,n,j]$ with $1\leq i<j<n$, $1\leq s\leq d_{\beta}$, we first define $\phi^{(1)}_{\beta,s}(F)$ by specializing :
\begin{equation*}
 \left\{
    \begin{aligned}
        x^{(\beta,s)}_{\ell\neq n,1}&\mapsto w_{\beta,s}-\tfrac{\ell}{2}, \\
       x^{(\beta,s)}_{n,1}&\mapsto w_{\beta,s}-\tfrac{n+1}{2} , \\
        x^{(\beta,s)}_{\ell\neq n,2}&\mapsto w'_{\beta,s}-\tfrac{2n+2-\ell}{2}.
    \end{aligned}
    \right.
\end{equation*}
According to wheel conditions \eqref{wheel-Y}, $\phi^{(1)}_{\beta,s}(F)$ is divisible by 
\begin{equation*}
  B_{\beta}=\prod^{n-1}_{\ell=j}(w_{\beta,s}-w'_{\beta,s}+n-\ell+2 )\prod^{n-2}_{\ell=j}(w_{\beta,s}-w'_{\beta,s}+n-\ell ).
\end{equation*}
Then the overall specialization $\phi_{\beta,s}(F)$ is defined by
\begin{equation}
\label{spe-C-2}
  \phi_{\beta,s}(F)\coloneqq \phi^{(2)}_{\beta,s}\left(\phi^{(1)}_{\beta,s}(F)\right) = \left. \frac{\phi^{(1)}_{\beta,s}(F)}{B_{\beta}}\right|_{w'_{\beta,s}\mapsto w_{\beta,s}}.
\end{equation}
For $\beta= [i,n,i]$ with $1\leq i<n$, we first define $\phi^{(1)}_{\beta,s}(F)$ by specializing:
\begin{equation*}
 \left\{
    \begin{aligned}
        x^{(\beta,s)}_{\ell\neq n,1}&\mapsto w_{\beta,s}-\tfrac{\ell}{2}, \\
       x^{(\beta,s)}_{n,1}&\mapsto w'_{\beta,s}-\tfrac{n+1}{2} , \\
        x^{(\beta,s)}_{\ell\neq n,2}&\mapsto w'_{\beta,s}-\tfrac{\ell}{2}.
    \end{aligned}
    \right.
\end{equation*}
According to wheel conditions \eqref{wheel-Y}, $\phi^{(1)}_{\beta,s}(F)$ is divisible by 
\begin{equation*}
  B_{\beta}=\big\{(w_{\beta,s}-w'_{\beta,s}+1 )(w_{\beta,s}-w'_{\beta,s}-1 )\big\}^{n-i-1}.
\end{equation*}
Then the overall specialization $\phi_{\beta,s}(F)$ is defined by
\begin{equation}
\label{spe-C-3}
  \phi_{\beta,s}(F)\coloneqq \phi^{(2)}_{\beta,s}\left(\phi^{(1)}_{\beta,s}(F)\right) = \left. \frac{\phi^{(1)}_{\beta,s}(F)}{B_{\beta}}\right|_{w'_{\beta,s}\mapsto w_{\beta,s}+1}.
\end{equation}

\item 
$D_n$-type: for $\beta \neq [i,n,j]\ (1\leq i<j\leq n-2)$, we define $\phi_{\beta,s}(F)$ by specializing:
\begin{equation}\label{spe-D-1}
 \left\{
    \begin{aligned}
        x^{(\beta,s)}_{\ell\neq n,1}&\mapsto w_{\beta,s}-\tfrac{\ell}{2} , \\
       x^{(\beta,s)}_{n,1}&\mapsto w_{\beta,s}-\tfrac{n-1}{2}.  
    \end{aligned}
    \right.
\end{equation}
For $\beta= [i,n,j]$ with $1\leq i<j\leq n-2$, we first define $\phi^{(1)}_{\beta,s}(F)$ by  specializing:
\begin{equation*}
 \left\{
    \begin{aligned}
        x^{(\beta,s)}_{\ell\neq n,1}&\mapsto w_{\beta,s}-\tfrac{\ell}{2}, \\
       x^{(\beta,s)}_{n,1}&\mapsto w'_{\beta,s}-\tfrac{n-1}{2} , \\
        x^{(\beta,s)}_{\ell\neq n-1\& n,2}&\mapsto w'_{\beta,s}-\tfrac{2n-2-\ell}{2}.
    \end{aligned}
    \right.
\end{equation*}
According to wheel conditions \eqref{wheel-Y}, $\phi^{(1)}_{\beta,s}(F)$ is divisible by 
\begin{equation*}
  B_{\beta}=\prod^{n-2}_{\ell=j}(w_{\beta,s}-w'_{\beta,s}-n+\ell+2 )(w_{\beta,s}-w'_{\beta,s}-n+\ell ).
\end{equation*}
Then, the overall specialization $\phi_{\beta,s}(F)$ is defined by:
\begin{equation}
\label{spe-D-2}
  \phi_{\beta,s}(F)\coloneqq \phi^{(2)}_{\beta,s}\left(\phi^{(1)}_{\beta,s}(F)\right) = \left. \frac{\phi^{(1)}_{\beta,s}(F)}{B_{\beta}} \right|_{w'_{\beta,s}\mapsto w_{\beta,s}}.
\end{equation}

\end{itemize}

For $\unl{d}\in \mathrm{KP}(\unl{k})$, the specialization map $\phi_{\underline{d}}(F)$ is defined by applying those 
separate maps $\phi_{\beta,s}$ in each group $\big\{x^{(\beta,s)}_{i,t}\big\}^{1\leq t\leq \nu_{\beta,i}}_{1\leq i\leq n}$ 
of variables (the result is independent of splitting):
\begin{equation*}
  \phi_{\underline{d}}\colon W^{(n)}_{\unl{k}}\longrightarrow
  \CC[\{w_{\beta,s}\}_{\beta\in\Delta^{+}}^{1\leq s\leq d_{\beta}}]^{\Sigma_{\unl{d}}},
\end{equation*}
and we extend it by zero to all other components $W^{(n)}_{\unl{\ell}}$ with $\unl{\ell}\ne \unl{k}$.

For any $h\in \mathcal{H}$, we define its \emph{degree} $\text{deg}(h)\in \NN^{\Delta^{+}}$ as the Kostant partition
$\unl{d}=\{d_{\beta}\}_{\beta\in\Delta^{+}}$ with $d_{\beta}=\sum_{r\in \mathbb{N}} h(\beta,r)\in \NN$ for all
$\beta\in \Delta^{+}$, and the \emph{grading} $\text{gr}(h)\in \NN^n$ so that
$\text{deg}(h)\in\text{KP}(\text{gr}(h))$. For any $\unl{k}\in\NN^{n}$ and $\unl{d}\in\text{KP}(\unl{k})$,
we define the following subsets of $\mathcal{H}$:
\begin{equation*}
  \mathcal{H}_{\underline{k}}\coloneqq \big\{h\in \mathcal{H} \, \big| \,  \text{gr}(h)=\underline{k}\big\}, \qquad
  \mathcal{H}_{\underline{k},\underline{d}}\coloneqq \big\{h\in \mathcal{H} \, \big| \, \text{deg}(h)=\underline{d}\big\}.
\end{equation*}
For any $h\in \mathcal{H}_{\unl{k},\unl{d}}$ and $\beta\in\Delta^{+}$, we consider the partition
\begin{equation}\label{eq:lambda-collection}
  \lambda_{h,\beta}=\big\{r_{\beta}(h,1)\geq \dots\geq r_{\beta}(h,d_{\beta})\big\}
\end{equation}
obtained by listing all integers $r\in\mathbb{N}$ with multiplicity $h(\beta,r)>0$ in the decreasing order. The key properties of $\phi_{\unl{d}}$ are summarized in the following lemmas.

Same as \cite[Example 6.13]{Tsy21} \cite[Lemma 5.5]{HT24} \cite[Lemmas 5.2--5.3]{HT26}, with the above modified specialization maps and \eqref{eq:psirv-1}--\eqref{eq:psirv-2}, we have
\begin{Lemma}
\label{phiyangianrs}
If $\mathfrak{g}$ is a  simple Lie algebra of classical type, then we have:
\begin{equation*}
  \phi_{\beta,1}(\Psi(e_{\beta}(r)))\doteq
  (w_{\beta,1}-\kappa_{\beta})^r \qquad \forall\, (\beta,r)\in\Delta^{+}\times \NN,
\end{equation*}
where $\kappa_{\beta}\in \ZZ[\tfrac{1}{2}]$ is given by
\[\kappa_{\beta}=\begin{cases}
\frac{i}{2}  & \text{for}\  \beta=[i,j], [i,n,j]\  \text{in any type, except }\beta=[i,n]\text{ in type }B_n, \\
\frac{n}{2} & \text{for}\  \beta=[i,n]\ \text{in type}\ B_n,\\
\frac{n-1}{2} & \text{for}\  \beta=[i,n,i]\ \text{in type}\ C_n.
\end{cases}\]
\end{Lemma}

The above modified specialization maps still satisfy the three main properties (cf. \cite{Tsy21,HT24,HT26}  mentioned above) listed below, as their proofs rely only on analyzing the $\zeta$-factors, unaffected by the above modification.

\begin{Lemma}\label{lem:spepsipbwbasis}
For  any $h\in \mathcal{H}_{\unl{k},\unl{d}}$, we have
\begin{equation}\label{eq:spephidEh}
  \phi_{\unl{d}}(\Psi(E_{h}))\doteq
\prod_{\beta,\beta'\in \Delta^+}^{\beta<\beta'}{G}_{\beta,\beta'}\cdot \prod_{\beta\in\Delta^{+}}{G}_{\beta}\cdot
  \prod_{\beta\in\Delta^{+}}\hat{P}_{\lambda_{h,\beta}},
\end{equation}
where ${G}_{\beta,\beta'}, {G}_{\beta}$ are products of linear factors $w_{\beta,s}-\ZZ[\tfrac{1}{2}]\cdot w_{\beta',s'}$ and independent of $h\in \mathcal{H}_{\unl{k},\unl{d}}$ (computed explicitly in the work cited above), while
\begin{equation}
\label{hlp-rat}
 \hat{P}_{\lambda_{h,\beta}}=\Sym_{\Sigma_{d_{\beta}}}
\left(\prod_{s=1}^{d_{\beta}}(w_{\beta,s}-\kappa_{\beta})^{r_{\beta}(h,s)}
  \prod_{1\leq s<r\leq d_{\beta}}\frac{w_{\beta,s}-w_{\beta,r}+\frac{(\beta,\beta)}{2}}{w_{\beta,s}-w_{\beta,r}}\right).
\end{equation}    
\end{Lemma}

\begin{Lemma}\label{lem:spefiltration}
 For any $h\in \mathcal{H}_{\unl{k},\unl{d}}$ and $\unl{d}'<\unl{d}$, cf.~\eqref{eq:KP-order}, we have $\phi_{\unl{d}'}(\Psi(E_{h}))=0$.   
\end{Lemma}

\begin{Lemma}\label{prop:mainpropertyspe}
 Let $W'_{\unl{k}}$ be the subspace of $W^{(n)}_{\unl{k}}$ spanned by $\{\Psi(E_{h})\}_{h\in \mathcal{H}_{\unl{k}}}$.  For any $F\in W^{(n)}_{\unl{k}}$ and $\unl{d}\in \text{\rm KP}(\unl{k})$, if $\phi_{\unl{d}'}(F)=0$ for all
$\unl{d}'\in \text{\rm KP}(\unl{k})$ such that $\unl{d}'<\unl{d}$, then there exists $F_{\unl{d}}\in W'_{\unl{k}}$
such that $\phi_{\unl{d}}(F)=\phi_{\unl{d}}(F_{\unl{d}})$ and $\phi_{\unl{d}'}(F_{\unl{d}})=0$ for all $\unl{d}'<\unl{d}$.
\end{Lemma}

\subsection{Integral form $\Y^>_n$ and its shuffle algebra realization}\label{subsection:integralformddd}

\subsubsection{$\ZZ[\tfrac{1}{2}]$-form of enveloping algebra}

Following \cite[Chapter VII]{Hum72}, we briefly recall the Kostant
$\ZZ$-form of the universal enveloping algebra $U(\mathfrak{g})$ for
a classical simple Lie algebra $\mathfrak{g}$ of type $A_n,B_n,C_n,D_n$.
We retain the Chevalley basis
$\{x_\beta\}_{\beta\in\Delta}\cup\{h_i\}_{i=1}^{n}$ introduced in (\ref{eq:chevalleybasis-A},  \ref{eq:chevalleybasis-BCD}). For $\beta\in\Delta^{+}$ and $t\in\NN$, we set
\[
x_\beta^{(t)}:=\frac{x_\beta^{t}}{t!}.
\]
The Kostant $\ZZ$-form $\U(\mathfrak{g})$ is the subring of
$U(\mathfrak{g})$ generated by
\[
\bigl\{\,x_\beta^{(t)},\; x_{-\beta}^{(t)}\;\big|\;
\beta\in\Delta^{+},\; t\in\NN \,\bigr\}.
\]
Let $\U^{>}(\mathfrak{g})$ be the subring generated by the divided
powers of the positive root vectors alone:
$\{x_\beta^{(t)}\mid \beta\in\Delta^{+},\;t\in\NN\}$. The following lemma collects the commutation formulas for divided
powers of positive root vectors, cf. \cite[(9.1.2)]{CW15}.
We write $\alpha,\beta$ for two linearly independent positive roots and
$\Gamma_{\alpha,\beta}=\{i\alpha+j\beta\mid i,j\in\ZZ^{+},\;i\alpha+j\beta\in\Delta^{+}\}$.

\begin{Lemma}\label{lem:commutation-divided-powers}
In $\U^{>}(\mathfrak{g})$, the following relations hold for all $r,t\in\NN$.
\begin{enumerate}[leftmargin=0.7cm]
\item If $\Gamma_{\alpha,\beta}=\{\alpha,\beta\}$, then
  \begin{equation}\label{eq:comm-case1}
  x_\beta^{(r)}x_\alpha^{(t)} = x_\alpha^{(t)}x_\beta^{(r)}.
  \end{equation}
\item If $\Gamma_{\alpha,\beta}=\{\alpha,\beta,\alpha+\beta\}$, then
  \begin{equation}\label{eq:comm-case2}
  x_\beta^{(r)}x_\alpha^{(t)}
  = \sum_{i=0}^{\min(r,t)} N_{\beta\alpha}^{i}\,
    x_\alpha^{(t-i)}x_\beta^{(r-i)}x_{\alpha+\beta}^{(i)},
  \end{equation}
  
\item If $\Gamma_{\alpha,\beta}=\{\alpha,\beta,\alpha+\beta,\alpha+2\beta\}$, then
  \begin{equation}\label{eq:comm-case3}
  x_\beta^{(r)}x_\alpha^{(t)}
  = \sum^{i+2j\leq r}_{i+j\leq t} 2^{-j} N_{\beta\alpha}^{i+j} N_{\alpha+\beta,\beta}^{j}\,
    x_\alpha^{(t-i-j)}x_\beta^{(r-i-2j)}x_{\alpha+\beta}^{(i)}x_{\alpha+2\beta}^{(j)}.
  \end{equation}
\item If $\Gamma_{\alpha,\beta}=\{\alpha,\beta,\alpha+\beta,2\alpha+\beta\}$, then
  \begin{equation}\label{eq:comm-case4}
  x_\beta^{(r)}x_\alpha^{(t)}
  = \sum^{i+2j\leq t}_{i+j\leq r} 2^{-j} N_{\beta\alpha}^{i+j} N_{\alpha+\beta,\alpha}^{j}\,
    x_{\alpha}^{(t-i-2j)}x_{\beta}^{(r-i-j)}x_{\alpha+\beta}^{(i)}x_{2\alpha+\beta}^{(j)}.
  \end{equation}

\end{enumerate}

Here $N_{\beta\alpha}$ is the Chevalley structure constant: $[x_\beta,x_\alpha]=N_{\beta\alpha}x_{\alpha+\beta}$, which is $\pm1$ or $\pm2$ in the Chevalley basis that we have chosen in subsection \ref{subsec:classicalliealgebra}.
\end{Lemma}

We now substitute the explicit Chevalley basis elements from 
\eqref{eq:chevalleybasis-A} and \eqref{eq:chevalleybasis-BCD} into the general 
commutation relations of Lemma~\ref{lem:commutation-divided-powers}. 
The goal is to obtain commutation relations for the divided powers 
\(F_{ij}^{(t)} := F_{ij}^t / t!\) and to show that 
all structure constants lie in \(\mathbb{Z}[\frac12]\).

For types \(A_n, C_n, D_n\), every positive root vector \(x_\beta\) is either 
\(F_{ij}\), $F_{ij'}$,  or  \(\frac12 F_{ii'}\) (in type \(C_n\) for \(2\varepsilon_i\)). 
Thus the substitution yields directly relations with coefficients in 
\(\mathbb{Z}[\frac12]\). For type \(B_n\), 
\(x_{\varepsilon_i} = \sqrt{2} F_{i,n+1}\). 
We must check that in every occurrence of \(\varepsilon_i\) in the root set 
\(\Gamma_{\alpha,\beta}\), the \(\sqrt{2}\) factors cancel appropriately. 
Take, for example, the case \(\alpha = \varepsilon_i-\varepsilon_j\), \(\beta = \varepsilon_j\) 
(so \(\alpha,\beta\) are linearly independent). Then
\[
\Gamma_{\alpha,\beta} = \{\alpha,\beta,\alpha+\beta,\alpha+2\beta\}
= \{\varepsilon_i-\varepsilon_j,\; \varepsilon_j,\; \varepsilon_i,\; \varepsilon_i+\varepsilon_j\}.
\]
Substituting \(x_\beta = \sqrt{2}F_{j,n+1}\), \(x_\alpha = F_{ij}\), 
\(x_{\alpha+\beta} = \sqrt{2}F_{i,n+1}\), \(x_{\alpha+2\beta} = F_{i,j'}\) (with \(j' = N+1-j\)) into  \eqref{eq:comm-case3} gives
\[
F_{j,n+1}^{(r)} F_{ij}^{(t)} = \sum_{i+2j\le r}^{\; i+j\le t} 2^{-2j} N^{i+j}_{\beta\alpha} N^{j}_{\alpha+\beta,\beta} \,
F_{ij}^{(t-i-j)} F_{j,n+1}^{(r-i-2j)} F_{i,n+1}^{(i)} F_{i,j'}^{(j)}.
\] 
All other cases involving \(\varepsilon_i\)  are verified similarly. Therefore, for all classical Lie algebras, the subalgebra  $\U^{>}_{\ZZ[\frac{1}{2}]}(\mathfrak{g})$ of $U^{>}(\fg)$ generated by the 
divided powers of the \(F_{ij}\) generators over $\ZZ[\frac12]$  provides a \(\mathbb{Z}[\frac12]\)-form of the universal enveloping algebra. Moreover, these commutation relations can be inverted, so that we can express every divided power of positive root vector as a $\ZZ[\tfrac{1}{2}]$-linear combination of
products of divided powers of simple root vectors. Explicitly, define for each positive root $\beta\in\Delta^+$ the element
\begin{equation}\label{eq:tildexbeta}
  \tilde{x}_\beta := 
\begin{cases}
F_{ij}, & \text{if } \beta = \varepsilon_i-\varepsilon_j\;(i< j),\\[2pt]
F_{ij'}, & \text{if } \beta = \varepsilon_i+\varepsilon_j\;(i<j),\\[2pt]
F_{i,n+1}, & \text{if } \beta = \varepsilon_i \text{ in type } B_n,\\[2pt]
 F_{ii'}, & \text{if } \beta = 2\varepsilon_i \text{ in type } C_n,
\end{cases}  
\end{equation}
and set $\tilde{x}_\beta^{(t)} := \tilde{x}_\beta^{\,t}/t!$ for $t\in\NN$. Then:

\begin{Proposition}\label{prop:kostant-form-simple-roots}
  Let $\mathfrak{g}$ be a  simple Lie algebra of classical type, with simple roots $\{\alpha_1,\dots,\alpha_n\}$.
  Denote by $\U^{>}_{\ZZ[\frac12]}(\fg)$ the $\ZZ[\frac12]$-subalgebra of
$U^{>}(\fg)$ generated by the divided powers $\tilde{x}_\beta^{(t)}\ (\beta\in\Delta^{+}, t\in\NN)$. Then $\U^{>}_{\ZZ[\frac12]}(\fg)$ is generated by the divided
  powers $\{\tilde{x}_{\alpha_i}^{(t)}\mid 1\leq i\leq n,\; t\in\NN\}$ over $\ZZ[\frac{1}{2}]$.
\end{Proposition}


\subsubsection{Integral form  $\Y^>_n$}\label{subsec:shuffleintegralform}

We shall now introduce  a $\ZZ[\tfrac{1}{2}]$-integral form $\Y_n^>$ of the Yangian $Y^>_n$. For $1\leq i\leq n$ and $r, t\in\NN$, define the {divided powers}
\begin{align}\label{def:divided power:simple root}
\e_{i,r}^{(t)}\coloneqq \frac{e_{i,r}^t}{t!},
\end{align}
and let $\Y_n^>$ be the $\ZZ[\frac{1}{2}]$-subalgebra of $Y_n^>$ generated by $\{\e_{i,r}^{(t)};~1\leq i\leq n, r,t\in\NN\}$.
Recalling the root vectors $\{e_\beta(r);~\beta\in\Delta^+,r\in\NN\}$ of (\ref{eq:B-root-vector-special}--\ref{eq:root-vector-standard}),  we also define their divided powers
\begin{align}\label{def:divided power:positive root}
\e_\beta(r)^{(t)}:=\frac{e_\beta(r)^t}{t!}\quad (t\in\NN).  
\end{align}
The following result is adapted from \cite[Proposition 1.2]{Tsy23}:
\begin{Proposition}\label{prop:dividedpoweresebetar}
For any $\beta\in\Delta^+$, $r,t\in\NN$, we have
   \[
  \e_\beta(r)^{(t)}\in\Y_n^>.    
   \]
\end{Proposition}
\begin{proof}
For any $\beta\in\Delta^{+}$, $r\in \NN$, consider the following map $\sigma_{\beta,r}\colon\ U^{>}(\fg)\longrightarrow Y^{>}_n$:
\begin{itemize}[leftmargin=0.7cm]
    \item For $\beta=[i,j]\ \text{or}\ [i,n,j]$, except for $\beta=[i,n]$ in type $B_n$, $\sigma_{\beta,r}$ is defined by sending (cf. \eqref{eq:tildexbeta})
    \[\tilde{x}_{\alpha_i}\mapsto e_{i,r},\qquad \tilde{x}_{\alpha_{\ell\neq i}}\mapsto e_{\ell,0}.\]
    \item For $\beta=[i,n]$ in type $B_n$, $\sigma_{\beta,r}$ is defined by sending
    \[\tilde{x}_{\alpha_n}\mapsto e_{n,r},\qquad \tilde{x}_{\alpha_{\ell\neq n}}\mapsto e_{\ell,0}.\]
    \item For $\beta=[i,n,i]$ in type $C_n$, $\sigma_{\beta,r}$ is defined by sending
    \[\tilde{x}_{\alpha_n}\mapsto e_{n,r},\qquad \tilde{x}_{\alpha_{\ell\neq n}}\mapsto e_{\ell,0}.\]
\end{itemize}
Then $\sigma_{\beta,r}$ is a $\CC$-algebra homomorphism, cf. (\ref{eq:Serre-relations}, \ref{eq:Drinfeld-rel2}), satisfying \[\sigma_{\beta,r}\left(\U^{>}_{\ZZ[\frac12]}(\fg))\right)\subset \Y^>_n,\]
 and $\sigma_{\beta,r}(\tilde{x}_{\beta})\doteq e_{\beta}(r)$\footnote{This follows from direct computations using \eqref{eq:commuformulaFij}}. Then it follows from $\tilde{x}_\beta^{(t)}\in \U^{>}_{\ZZ[\frac12]}(\fg)$ that $\e_\beta(r)^{(t)}\in\Y_n^>$.
\end{proof}

Similar to \eqref{eq:pbwd-yangian}, for any function $h\in\mathcal{H}$, consider the ordered monomials

\begin{equation}\label{eq:pbwd-integralyangian}
\E_{h}\, =\prod_{(\beta,r)\in\Delta^{+}\times\mathbb{N}}\limits^{\rightarrow} \e_{\beta}(r)^{(h(\beta,r))}.
\end{equation}
Then we have the   PBW theorem for the integral form $\Y^>_n$:
 
\begin{Theorem}\label{thm:pbwd-integralyangian}
The elements $\{\E_{h}\}_{h\in \mathcal{H}}$ of \eqref{eq:pbwd-integralyangian} form a basis of the $\ZZ[\frac{1}{2}]$-module  $\Y^>_n$.
\end{Theorem}

The proof of Theorem \ref{thm:pbwd-integralyangian} relies on the  shuffle algebra realization of $\Y^>_n$. Let us now introduce an integral form $\W^{(n)}$ for the rational shuffle algebra $W^{(n)}$. Consider the $\ZZ[\frac{1}{2}]$-submodule $\W^{(n)}_{\unl{k}}$ of $W^{(n)}_{\unl{k}}$ consisting of rational functions $F$ satisfying the following  condition: If $f$ denotes the numerator of $F$ from \eqref{polecondition}, then
\begin{equation}
\label{integralformY-1}
  f\in \ZZ[\tfrac{1}{2}][\{x_{i,r}\}_{1\leq i\leq  n}^{1\leq r\leq k_{i}}]^{\Sigma_{\underline{k}}}.
\end{equation}
Let $\W^{(n)}\coloneqq \bigoplus_{\unl{k}\in\NN^{n}}\W^{(n)}_{\unl{k}}$. Then, we have:
\begin{Lemma}\label{lemma:inclusion}
 $\Psi(\Y^>_n)\subset \W^{(n)}$.
\end{Lemma}
\begin{proof}
For any $m\in\NN$, $ i_{1},\dots,i_{m}\in I$, $r_{1},\dots,r_{m}\in\NN$, and $t_{1},\dots,t_{m}\in\NN$, let
\[
  F\coloneqq \Psi\big(\e^{(t_{1})}_{i_{1},r_{1}}\cdots \e^{(t_{m})}_{i_{m},r_{m}}\big),
\]
and $f$ be the numerator of $F$, then the validity of \eqref{integralformY-1} for $f$ follows from Lemmas \ref{k times of r-powerof x1},  so   $\Psi(\Y^>_n)\subset \W^{(n)}$. 
\end{proof}


\begin{Lemma}\label{lem:span}
For any $F\in \W^{(n)}_{\unl{k}}$ and $\unl{d}\in \text{\rm KP}(\unl{k})$, if $\phi_{\unl{d}'}(F)=0$ for all
$\unl{d}'\in \text{\rm KP}(\unl{k})$ such that $\unl{d}'<\unl{d}$, then there exists $F_{\unl{d}}\in \Psi(\Y^>_n)$
such that $\phi_{\unl{d}}(F)=\phi_{\unl{d}}(F_{\unl{d}})$ and $\phi_{\unl{d}'}(F_{\unl{d}})=0$ for all $\unl{d}'<\unl{d}$.
\end{Lemma}
\begin{proof}
For any $F\in \W^{(n)}_{\unl{k}}$, from the definition of specialization maps \eqref{spe-A}--\eqref{spe-D-2}, and  \eqref{integralformY-1}, we know 
\[\phi_{\unl{d}}(F)\in\ZZ[\tfrac{1}{2}][\{w_{\beta,s}\}_{\beta\in\Delta^{+}}^{1\leq s\leq d_{\beta}}]^{\Sigma_{\unl{d}}}.\] 
From the proofs of \cite[Lemma 3.21]{Tsy21}\cite[Proposition 4.6]{HT24}\cite[Propositions 3.8, 4.6]{HT26}, if $F\in \W^{(n)}_{\unl{k}}$ satisfies that $\phi_{\unl{d}'}(F)=0$ for all
$\unl{d}'\in \text{\rm KP}(\unl{k})$ such that $\unl{d}'<\unl{d}$, then $\phi_{\unl{d}}(F)$ is divisible by
\[\prod_{\beta,\beta'\in \Delta^+}^{\beta<\beta'}G_{\beta,\beta'}\cdot
  \prod_{\beta\in\Delta^{+}} G_{\beta},\]
 which are the same factors that appear in \eqref{eq:spephidEh}. Moreover, for any $\E_h$ of \eqref{eq:pbwd-integralyangian},
 
 \begin{equation}
\label{spe-bfEh}
  \phi_{\unl{d}}(\Psi(\E_{h}))\doteq
  \prod_{\beta,\beta'\in \Delta^+}^{\beta<\beta'}G_{\beta,\beta'}\cdot
  \prod_{\beta\in\Delta^{+}} G_{\beta}\cdot
  \prod_{\beta\in\Delta^{+}}  P_{\lambda_{h,\beta}},
\end{equation}    
where $P_{\lambda_{h,\beta}}\coloneqq \frac{1}{\mathrm{mul}(\lambda_{h,\beta})}\cdot \hat{P}_{\lambda_{h,\beta}}$, cf. \eqref{hlp-rat}.  From Subsection \ref{subs:zbasis}, we know $\{\prod_{\beta\in\Delta^{+}}  P_{\lambda_{h,\beta}}\}_{h\in \mathcal{H}_{\unl{k},\unl{d}}}$ forms  a $\ZZ[\frac{1}{2}]$-basis of  $\ZZ[\frac{1}{2}][\{w_{\beta,s}\}_{\beta\in\Delta^{+}}^{1\leq s\leq d_{\beta}}]^{\Sigma_{\unl{d}}}$, then it follows from Lemma  \ref{lem:spefiltration}  that we can find  $F_{\unl{d}}\in \Psi(\Y^>_n)$ such that $\phi_{\unl{d}}(F)=\phi_{\unl{d}}(F_{\unl{d}})$ and $\phi_{\unl{d}'}(F_{\unl{d}})=0$ for all $\unl{d}'<\unl{d}$.
\end{proof}

 Lemma \ref{lem:span} implies  that we have the opposite inclusion $\W^{(n)}\subset \Psi(\Y^{>}_n)$. Moreover, from the proof of Lemmas \ref{lemma:inclusion}--\ref{lem:span}, we obtain that $\{\Psi(\E_h)\}_{h\in \mathcal{H}}$ form a basis of $\W^{(n)}$. Hence:

\begin{Theorem}\label{ShuffleIY-A}
(1) The $\CC$-algebra isomorphism $\Psi\colon Y_n^> \stackrel{\sim}{\longrightarrow}W^{(n)}$ of \eqref{eq:shufflerealizationyangian} gives rise to 
a $\ZZ[\frac{1}{2}]$-algebra isomorphism $\Psi\colon \Y_n^> \stackrel{\sim}{\longrightarrow} \W^{(n)}$.

\smallskip
\noindent
(2) Theorem {\rm \ref{thm:pbwd-integralyangian}} holds.
\end{Theorem}

The integral form $\Y_n^>$ will be used to study modular Yangians in Section \ref{sec:modularshufflerealization}.




\section{Modular Yangians: PBW theorem and the $p$-center in the Drinfeld presentation}\label{sec:modular-yangian-structure}

\subsection{RTT and Drinfeld presentations of Modular Yangians}\label{subsec:modular-RTT-Drinfeld}

Let $\kk$ be an algebraically closed field of characteristic $p>3$, and
let $\fg$ be of type $A_n$, $B_n$, $C_n$ or $D_n$.  The modular Yangian
was studied in type $A$ by Brundan--Topley \cite{BT18},
and in  types $B,D$ by Chang--Hu \cite{CH25}.
The isomorphism between the RTT and Drinfeld
presentations of Yangians in characteristic zero was established by Brundan--Kleshchev \cite{BK05} and Jing--Liu--Molev \cite{JLM18} by using Gauss
decomposition. 
Its proof remains valid over $\kk$ for $p>3$, see
\cite[Theorem~3.1 and Remark~3.10]{CH25}. 
Thus we may use the Drinfeld presentation of the modular Yangian
$Y_n(\kk)$ in all classical types.  Its positive subalgebra $Y_n^>(\kk)$ is the $\kk$-algebra
generated by $\{e_{i,r}\mid 1\leq i\leq n,\ r\in\NN\}$ with the relations
\eqref{eq:Drinfeld-rel1}--\eqref{eq:Drinfeld-rel2}.  We use throughout
the root vectors fixed in Subsection~\ref{subsec:Drinfeldpre}.

We write $Y_n^{\mathrm{RTT}}(\kk)$ for the RTT presentation of modular Yangian.  Recall that  $N=n+1,2n+1,2n,2n$ in types $A_n,B_n,C_n,D_n$, respectively, and set
\begin{equation}\label{eq:RTT-generating-matrix}
 t_{ij}(u)=\delta_{ij}+\sum_{s\geq1}t_{ij}^{(s)}u^{-s},
 \qquad
 T(u)=\sum_{i,j=1}^N E_{ij}\otimes t_{ij}(u).
\end{equation}
In type $A$, let $P=\sum_{i,j}E_{ij}\otimes E_{ji}$ and
$R(u)=1-P/u$.  In types $B,C,D$,  put
\[
 P=\sum_{i,j}E_{ij}\otimes E_{ji},
 \qquad
 Q=\sum_{i,j}\theta_i\theta_jE_{ij}\otimes E_{i'j'},
\]
where $\theta_i=1$ in the orthogonal cases and
$\theta_i=1$ for $i\leq n$, $\theta_i=-1$ for $i>n$ in type $C$.
Then
\[
 R(u)=1-\frac{P}{u}+\frac{Q}{u-\kappa},
 \qquad
 \kappa=
 \begin{cases}
  N/2-1,&\fg=\mathfrak{so}_N,\\
  N/2+1,&\fg=\mathfrak{sp}_N.
 \end{cases}
\]
The extended RTT Yangian is generated by the coefficients in
\eqref{eq:RTT-generating-matrix}, subject to
\begin{equation}\label{eq:modular-RTT-relation}
 R(u-v)T_1(u)T_2(v)=T_2(v)T_1(u)R(u-v).
\end{equation}
For type $A$, the RTT relation
\eqref{eq:modular-RTT-relation} defines
$Y(\mathfrak{gl}_N)$.  For every
$f(u)\in 1+u^{-1}\kk[[u^{-1}]]$, let
\[
 \mu_f:Y(\mathfrak{gl}_N)\longrightarrow Y(\mathfrak{gl}_N),
 \qquad
 \mu_f(T(u))=f(u)T(u).
\]
Then the  Yangian  in type $A$ is the fixed-point subalgebra
\begin{equation*}
 Y_n^{\mathrm{RTT}}(\kk)
 =
 \{x\in Y(\mathfrak{gl}_N)\mid
   \mu_f(x)=x\text{ for all }f(u)\},
\end{equation*}
 see
\cite[(6.1)]{BT18}. For types $B,C,D$, the Yangian $Y_n^{\mathrm{RTT}}(\kk)$ is obtained by imposing
\begin{equation*}
 T^t(u+\kappa)T(u)=1,
\end{equation*}
where $t$ is transpose with respect to the defining bilinear form, see
\cite[(2.10)--(2.17)]{JLM18} and \cite[Subsection~3.1]{CH25}. The Gauss decomposition is
\begin{equation*}
 T(u)=F(u)H(u)E(u),
 \qquad
 H(u)=\operatorname{diag}(h_1(u),\ldots,h_N(u)),
\end{equation*}
where $F(u)$ and $E(u)$ are lower and upper unitriangular, respectively.
For $i<j$, write
\begin{equation*}
 E(u)_{ij}=e_{ij}(u)=\sum_{s\geq1}e_{ij}^{(s)}u^{-s}.
\end{equation*}
 In particular, for $\beta\in\Delta^+$ define
\begin{equation}\label{eq:positive-root-matrix-pair}
 (a_\beta,b_\beta)=
 \begin{cases}
  (i,j),&\beta=\varepsilon_i-\varepsilon_j,\\
  (i,j'),&\beta=\varepsilon_i+\varepsilon_j,\\
  (i,n+1),&\beta=\varepsilon_i\text{ in type }B_n,\\
  (i,i'),&\beta=2\varepsilon_i\text{ in type }C_n,
 \end{cases}
\end{equation}
and put
\begin{equation}\label{eq:root-indexed-Gaussian-generator}
 e_\beta^{\mathrm{RTT}}(u):=e_{a_\beta b_\beta}(u)
 =\sum_{s\geq1}e_\beta^{(s)}u^{-s}.
\end{equation}
These are the positive Gaussian generators of
\cite[Section~4]{JLM18} and \cite[Subsection~3.3]{CH25}.
The positive RTT subalgebra $Y_n^{\mathrm{RTT},>}(\kk)$ is generated by
the coefficients of the upper Gaussian series corresponding to positive roots.

Let $\Phi:Y_n^{\mathrm{Dr}}(\kk)\stackrel{\sim}{\longrightarrow}
 Y_n^{\mathrm{RTT}}(\kk)$ be the RTT--Drinfeld isomorphism, and  $e_i(u)=\sum_{r\geq0}e_{i,r}u^{-r-1}$ be the simple positive
Drinfeld currents.  For $i<n$, it is given by
\begin{equation}\label{eq:nonfinal-Gaussian-formulas}
 \Phi(e_i(u))=e_{i,i+1}\left(u-\frac{i-1}{2}\right),
\end{equation}
and for $i=n$ by
\begin{equation}\label{eq:last-node-Gaussian-formulas}
\begin{aligned}
 A_n:&\quad
 \Phi(e_n(u))=e_{n,n+1}\left(u-\frac{n-1}{2}\right),\\
 B_n:&\quad
 \Phi(e_n(u))=e_{n,n+1}\left(u-\frac{n-1}{2}\right),\\
 C_n:&\quad
 \Phi(e_n(u))=\frac12e_{n,n+1}\left(u-\frac n2\right),\\
 D_n:&\quad
 \Phi(e_n(u))=e_{n-1,n+1}\left(u-\frac{n-2}{2}\right).
\end{aligned}
\end{equation}
Thus $\Phi$ restricts to an isomorphism
\begin{equation}\label{eq:positive-RTT-Drinfeld-isomorphism}
 Y_n^>(\kk)\stackrel{\sim}{\longrightarrow}
 Y_n^{\mathrm{RTT},>}(\kk).
\end{equation}

Equations \eqref{eq:nonfinal-Gaussian-formulas} and
\eqref{eq:last-node-Gaussian-formulas} are  the formulas following
\cite[(1.4)]{JLM18}, they are valid in the modular setting for $p>3$.

We shall also use  the translation automorphisms.  For $c\in\kk$, the assignment
\begin{equation}
   \eta_c(T(u))=T(u-c) 
\end{equation}
gives an automorphism of the RTT Yangian. In terms of Gaussian generators, we have
\[
 \eta_c(e_{ij}(u))=e_{ij}(u-c).
\]

\subsection{The PBW theorem in the Drinfeld presentation}\label{subsec:modular-pbw}

We use the following loop filtrations for $Y^{>}_n(\kk)$ and $Y_n^{\mathrm{RTT},>}(\kk)$ respectively:
\begin{equation*}
  \deg e_{i,r}=r,\qquad \deg e_{ij}^{(s)}=s-1.
\end{equation*}
 The PBW theorem of the RTT Yangian identifies the associated graded algebra with the
enveloping algebra of the current algebra.  On the positive part, we have
\begin{equation}\label{eq:gr-RTT-positive}
 \operatorname{gr}Y_n^{\mathrm{RTT},>}(\kk)
 \simeq U(\fn^+[t]).
\end{equation}
Under \eqref{eq:gr-RTT-positive}, by \cite[Lemma~4.2]{BT18} and
\cite[Remark~3.2(2)]{CH25}, we have
\begin{equation}\label{eq:gr-root-indexed-Gaussian-generator}
 \operatorname{gr}_r e_\beta^{(r+1)}
 =\tilde{x}_\beta t^r,
\end{equation}
cf. \eqref{eq:tildexbeta}. We now express  PBW theorem of $Y^{>}_n(\kk)$ in the
Lyndon root vectors fixed in Subsection~\ref{subsec:Drinfeldpre}. The following elementary filtered lemma will be used at the end of the argument.

\begin{Lemma}\label{lem:filtered-lifting-pbw}
Let $A=\bigcup_{d\geq0}F_dA$ be a nonnegatively filtered
$\kk$-algebra, where $F_{-1}A=0$.  Suppose that
$\{m_\lambda\}_{\lambda\in\Lambda}$ is a family of elements of $A$ whose
leading symbols form a $\kk$-basis of $\operatorname{gr}A$.  Then
$\{m_\lambda\}_{\lambda\in\Lambda}$ is a $\kk$-basis of $A$.
\end{Lemma}

The following proposition records the compatibility of the RTT--Drinfeld
isomorphism with the loop filtrations. Similar as \eqref{eq:doteqmean},  we write $x\doteq y$ if $x$ is a nonzero scalar multiple of $y$.

\begin{Proposition}\label{prop:gaussian-lyndon-triangular}
For every $\beta\in\Delta^+$ and $r\in\NN$, we have
\begin{equation}\label{eq:gr-Lyndon-root-vector}
 \operatorname{gr}_r\Phi(e_\beta(r))
 \doteq \tilde{x}_\beta t^r.
\end{equation}

\end{Proposition}

\begin{proof}
For a simple root, use
\eqref{eq:nonfinal-Gaussian-formulas}--\eqref{eq:last-node-Gaussian-formulas}.
If $a\in\kk$ and
$e_{ij}(u)=\sum_{s\geq1}e_{ij}^{(s)}u^{-s}$, then
\[
 e_{ij}(u-a)
 =\sum_{s\geq1}e_{ij}^{(s)}(u-a)^{-s}.
\]
The coefficient of $u^{-r-1}$ in the right hand side is
\[
 e_{ij}^{(r+1)}+
 \sum_{1\leq s\leq r}\binom{r}{s-1}a^{r+1-s}e_{ij}^{(s)}.
\]
Every term in the sum has loop degree at most $r-1$.  Consequently, for the
matrix pair $(a_i,b_i)$ displayed in \eqref{eq:positive-root-matrix-pair},
\begin{equation}\label{eq:simple-triangular}
 \Phi(e_{i,r})\doteq e_{a_i b_i}^{(r+1)}
 \pmod{F_{r-1}Y_n^{\mathrm{RTT}}(\kk)},
\end{equation}
 By \eqref{eq:gr-root-indexed-Gaussian-generator}, we have $\operatorname{gr}_r\Phi(e_{i,r})
 \doteq \tilde{x}_{\alpha_i}t^r$.  Recall that in any filtered associative algebra, if
$a\in F_dA$ and $b\in F_eA$, then
\begin{equation}\label{eq:symbol-of-commutator}
 \operatorname{gr}_{d+e}[a,b]
 =[\operatorname{gr}_d a,\operatorname{gr}_e b].
\end{equation}
If $e_{\beta}(r)$ is of the form  $\eqref{eq:root-vector-standard}$,  we obtain
\[
 \operatorname{gr}_r\Phi(e_\beta(r))
 \doteq 
 [\cdots[[\tilde{x}_{\alpha_{i_1}}t^r,
 \tilde{x}_{\alpha_{i_2}}],\tilde{x}_{\alpha_{i_3}}],\ldots,
 \tilde{x}_{\alpha_{i_\ell}}]\doteq \tilde{x}_\beta t^r.
\]
For $\beta=\epsilon_i$ of type $B$, and $e_{\beta}(r)$ is of the form \eqref{eq:B-root-vector-special}, we have
\[
 \operatorname{gr}_r\Phi(e_{\beta}(r))
 \doteq  [\cdots[[\tilde{x}_{\alpha_i},\tilde{x}_{\alpha_{i+1}}],
 \tilde{x}_{\alpha_{i+2}}],\ldots,
 \tilde{x}_{\alpha_n}t^r]\doteq \tilde{x}_\beta t^r.
\]
For $\beta=2\varepsilon_i$ with $i<n$ in type $C$, where
$e_{\beta}(r)$ is of the form \eqref{eq:C-root-vector-special}, put
\[
 A_i=[\cdots[[e_{i,0},e_{i+1,0}],e_{i+2,0}],\ldots,e_{n-1,0}].
\]
Then $\operatorname{gr}_0 \Phi(A_i)\doteq
\tilde{x}_{\varepsilon_i-\varepsilon_n}$, and
\[
 \operatorname{gr}_r\Phi(e_{\beta}(r))
 \doteq \,[\tilde{x}_{\varepsilon_i-\varepsilon_n},
 [\tilde{x}_{\varepsilon_i-\varepsilon_n},
 \tfrac12\tilde{x}_{2\varepsilon_n}t^r]]
 \doteq \tilde{x}_{\beta}t^r.
\]
This completes the proof.
\end{proof}

We can now prove the  PBW theorem of $Y^{>}(\kk)$.

\begin{Theorem}\label{thm:yangian-basisoverk}
Let $\kk$ be of characteristic $p>3$.  The elements (cf. \eqref{eq:pbwd-yangian})
\[
 \{E_h\}_{h\in\mathcal H},
 \qquad
 E_h=\prod_{(\beta,r)\in\Delta^+\times\NN}^{\longrightarrow}
 e_\beta(r)^{h(\beta,r)},
\]
form a $\kk$-basis of $Y_n^>(\kk)$.
\end{Theorem}

\begin{proof}
Give $Y_n^>(\kk)$ the loop filtration $\deg e_{i,r}=r$, and denote the
symbol of $e_{i,r}$ by $\overline e_{i,r}$.  As in the proofs of RTT presentations, see \cite[Theorem~4.1 and Lemma~4.2]{BT18} and
\cite[Subsection~3.2]{CH25}, we have a  surjection
\begin{equation}\label{eq:current-to-gr-Drinfeld}
 \pi:U(\fn^+[t])\twoheadrightarrow
 \operatorname{gr}Y_n^>(\kk),
 \qquad
 \tilde{x}_{\alpha_i}t^r\longmapsto\overline e_{i,r}.
\end{equation}
The map $\Phi$ is filtered by \eqref{eq:simple-triangular}, from \eqref{eq:positive-RTT-Drinfeld-isomorphism}  and \eqref{eq:gr-RTT-positive}, we obtain the associated  map 
\[
 \operatorname{gr}\Phi:
 \operatorname{gr}Y_n^>(\kk)\longrightarrow U(\fn^+[t]),
 \qquad
 \overline e_{i,r}\longmapsto
 \lambda_i\tilde{x}_{\alpha_i}t^r,\ \lambda_i\neq 0.
\]
The composite $\operatorname{gr}\Phi\circ\pi$ is  invertible, hence $\pi$ is an isomorphism and
\begin{equation}\label{eq:gr-positive-Drinfeld}
 \operatorname{gr}Y_n^>(\kk)=U(\fn^+[t]).
\end{equation}
The ordered monomials in
$\tilde{x}_\beta t^r$ form a basis of $U(\fn^+[t])$.  Proposition~\ref{prop:gaussian-lyndon-triangular} identifies the leading
symbol of $E_h$ with a nonzero scalar multiple of the corresponding PBW
monomial in $U(\fn^+[t])$, then    Lemma~\ref{lem:filtered-lifting-pbw} completes the proof.
\end{proof}

\subsection{The Drinfeld $p$-center}\label{subsec:modular-pcenter}

We now transfer the known Gaussian $p$-central elements to the Lyndon
root vectors. For the orthogonal Yangians, recall the independent index sets used in
\cite[Subsection~3.5]{CH25}.  If $N=2n+1$, put
\begin{equation}\label{eq:independent-Gaussian-indices-B}
\begin{split}
 \mathcal I_B={}&
 \bigl\{(a,b)\mid 1\leq a,b\leq N,\ a<b<a'\bigr\}
 \setminus\bigl\{(a,n+1)\mid 1\leq a<n\bigr\}\\
 &{}\cup\bigl\{(n+1,b)\mid n'<b\leq1'\bigr\},
\end{split}
\end{equation}
and, if $N=2n$, put
\begin{equation}\label{eq:independent-Gaussian-indices-D}
 \mathcal I_D=
 \bigl\{(a,b)\mid 1\leq a,b\leq N,\ a<b<a'\bigr\}.
\end{equation}

\begin{Proposition}\label{prop:RTT-Gaussian-pcentral}
Let $s\geq1$.
\begin{enumerate}[label={\rm(\arabic*)},leftmargin=0.7cm]
 \item In type $A$, for every $1\leq a<b\leq N$,
 \[
  (e_{ab}^{(s)})^p\in Z\bigl(Y_n^{\mathrm{RTT}}(\kk)\bigr).
 \]

 \item In type $B$ (respectively, type $D$), for every
 $(a,b)\in\mathcal I_B$ (respectively,
 $(a,b)\in\mathcal I_D$),
 \[
  (e_{ab}^{(s)})^p\in Z\bigl(Y_n^{\mathrm{RTT}}(\kk)\bigr).
 \]
\end{enumerate}
\end{Proposition}

\begin{proof}
Part~(1) is \cite[Theorem 6.6]{BT18} and Part~(2) follows from
\cite[Theorem~4.11]{CH25}.
\end{proof}

\begin{Proposition}\label{prop:exact-Gaussian-Lyndon-comparison}
Suppose that $\fg$ is of type $A$, $B$, or $D$.  For every
$\beta\in\Delta^+$ and $r\in\NN$, there are a matrix pair
$(A_\beta,B_\beta)$, a shift $c_\beta\in\kk$ such that
\begin{equation}\label{eq:exact-Gaussian-Lyndon-comparison}
 \Phi(e_\beta(r))
 \doteq 
 \eta_{c_\beta}(e_{A_\beta B_\beta}^{(r+1)}).
\end{equation}
They are given by the following table:
\begin{equation}\label{eq:tableabbeta}
\begin{array}{c|c|c|c}
\text{type}&\beta&(A_\beta,B_\beta)&c_\beta\\ \hline
A_n &[i,j]&(i,j+1)&(i-1)/2\\
B_n &[i,j],\ j<n&(i,j+1)&(i-1)/2\\
B_n &[i,n],\ i<n&(n+1,i')&(n-2)/2\\
B_n &\varepsilon_n=[n]&(n,n+1)&(n-1)/2\\
B_n &[i,n,j],\ i<j\leq n&(i,j')&(i-1)/2\\
D_n &[i,j],\ j<n&(i,j+1)&(i-1)/2\\
D_n &[i,n]&(i,n')&(i-1)/2\\
D_n &[i,n,j],\ i<j<n&(i,j')&(i-1)/2
\end{array}
\end{equation}
\end{Proposition}

\begin{proof}
 For type $A$ case, we have by \cite[(4.9)]{BT18} that
\begin{equation}\label{eq:type-A-Gaussian-recursion}
 e_{a,b+1}^{(s)}=[e_{ab}^{(s)},e_b^{(1)}].
\end{equation}
For types $B$ and $D$, we have
\begin{equation}\label{eq:reflected-Gaussian-recursion}
 e_{a,b'}^{(s)}=-[e_{a,b'-1}^{(s)},e_b^{(1)}],
\end{equation}
together with \eqref{eq:type-A-Gaussian-recursion}, see
\cite[Lemma~5.15]{JLM18}. 

For $\beta=[i_1,\dots,i_{\ell}]$, not $[i,n]$ of type $B$, we have 
\[e_{\beta}(r)=[\cdots[[e_{i_1,r},e_{i_2,0}],e_{i_3,0}],\cdots,e_{i_\ell,0}].\]
By \eqref{eq:nonfinal-Gaussian-formulas}--\eqref{eq:last-node-Gaussian-formulas}, there is $c\in \kk$ such that $\Phi(e_{i_1,r})=\eta_c(e^{(r+1)}_{i_1,i_1+1})$ and $\Phi(e_{b,0})=e^{(1)}_{b,b+1}=\eta_c(e^{(1)}_{b,b+1})$.
Applying
\eqref{eq:type-A-Gaussian-recursion}--\eqref{eq:reflected-Gaussian-recursion} successively gives
\eqref{eq:exact-Gaussian-Lyndon-comparison}.

For $\beta=[i,n]$ of type $B$, put
\[
 g_j(u)=e_{(j+1)',j'}(u)\quad(1\leq j<n),
 \qquad
 g_n(u)=e_{n+1,n+2}(u).
\]
Let $\kappa=n-\frac12$,  \cite[Proposition~5.7]{JLM18} gives
\begin{equation}\label{eq:B-lower-right-simple-Gaussians}
 g_j(u)=-e_j(u+\kappa-j)\quad(j<n),
 \qquad
 g_n(u)=-e_n\left(u-\frac12\right).
\end{equation}
In particular,
\[
 \Phi(e_{j,0})=-g_j^{(1)}\quad(j<n),
 \qquad
 \Phi(e_n(u))=-g_n\left(u-\frac{n-2}{2}\right).
\]
We rewrite \eqref{eq:B-root-vector-special} as
\begin{equation}\label{eq:B-short-right-nested}
 e_{[i,n]}(r)\doteq
 [e_{i,0},[e_{i+1,0},\ldots,
 [e_{n-1,0},e_{n,r}]\ldots]].
\end{equation}
It follows from \eqref{eq:type-A-Gaussian-recursion} that
\[
 [g_i^{(1)},[g_{i+1}^{(1)},\ldots,
 [g_{n-1}^{(1)},g_n^{(s)}]\ldots]]
 \doteq e_{n+1,i'}^{(s)}.
\]
Combining this identity with \eqref{eq:B-short-right-nested} yields
\begin{equation}\label{eq:B-short-exact-Gaussian-comparison}
 \Phi(e_{[i,n]}(r))
 \doteq \eta_{(n-2)/2}(e_{n+1,i'}^{(r+1)})
 \qquad(i<n).
\end{equation}
The table \eqref{eq:tableabbeta} is clear from the above proof.
\end{proof}

Combining Propositions \ref{prop:RTT-Gaussian-pcentral}--\ref{prop:exact-Gaussian-Lyndon-comparison}, we obtain the following theorem:
\begin{Theorem}\label{thm:Lyndon-root-pcentral}
Let $\fg$ be of type $A_n$, $B_n$, or $D_n$.  For every
$\beta\in\Delta^+$ and $r\in\NN$,
\begin{equation}\label{eq:Lyndon-root-pcentral}
 e_\beta(r)^p\in Z(Y_n(\kk)).
\end{equation}
In particular, $e_\beta(r)^p\in Z(Y_n^>(\kk))$.
\end{Theorem}

\begin{Remark}\label{rem:type-C-pcenter-conjecture}
For type $C$, 
under the assumption $p>3$,
it is stated in \cite[Remark~3.10]{CH25} that
the Drinfeld presentation of the symplectic Yangian is exactly the same as the one over the complex field.
However,
to describe the $p$-center, we need to consider the Drinfeld generators $e_{i,i'}(u)$.
 It seems to be hard to find a suitable automorphism
 which sends the $e_{i,i'}(u)$ to some simple Gaussian generator.
Here, we formulate the type-$C$ statement as the following
conjecture:
\begin{equation}\label{eq:type-C-Lyndon-pcentral-conjecture}
 e_\beta(r)^p\in Z(Y_n(\kk))
 \quad\text{for every }\beta\in\Delta^+,\ r\in\NN,
 \quad \fg=\mathfrak{sp}_{2n}.
\end{equation}
Once \eqref{eq:type-C-Lyndon-pcentral-conjecture} is
proved, Corollary \ref{cor:positive-pcenter-free-module} below also holds for type $C$ Yangians.
\end{Remark}

\begin{Corollary}\label{cor:positive-pcenter-free-module}
Let $\fg$ be of type $A_n$, $B_n$, or $D_n$.  We have
\begin{equation}\label{eq:positive-Lyndon-pcenter}
 Z_p^>(Y_n(\kk))
 :=\kk[e_\beta(r)^p\mid \beta\in\Delta^+,\ r\in\NN]
\end{equation}
is a free polynomial algebra. 
Moreover,
$Y_n^>(\kk)$ is a free module over $Z_p^>(Y_n(\kk))$ with basis
\begin{equation}\label{eq:p-restricted-PBW-basis}
 \left\{
 \prod_{(\beta,r)\in\Delta^+\times\NN}^{\longrightarrow}
 e_\beta(r)^{h(\beta,r)}
 \ \middle|\ 
 0\leq h(\beta,r)<p,\quad h\text{ has finite support}
 \right\}.
\end{equation}
\end{Corollary}

\begin{proof}
By Theorem~\ref{thm:yangian-basisoverk}, distinct monomials in the
elements $e_\beta(r)^p$ are linearly independent, so
\eqref{eq:positive-Lyndon-pcenter} is a polynomial algebra.  
The rest of the proof is the same as \cite[Corollary 5.13]{BT18}.
\end{proof}

For every classical type, let $I_p(Y_n^>(\kk))$ be the two-sided ideal
generated by the elements $e_\beta(r)^p$, and set
\begin{equation}\label{eq:defresyangian}
 Y_n^{>,[p]}:=Y_n^>(\kk)/I_p(Y_n^>(\kk)).
\end{equation}
In types $A$, $B$, and $D$, Theorem~\ref{thm:Lyndon-root-pcentral}
shows that this is a quotient by central elements.
Following \cite{GT21}, we call it the
\emph{restricted positive Yangian}.  
Corollary~
\ref{cor:positive-pcenter-free-module} shows that its basis consists
of the images of the $p$-restricted ordered monomials. 
In type $C$,
we still use the same notation for the two-sided ideal quotient, without assuming centrality.
Section~\ref{sec:modularshufflerealization} will
establish its $p$-restricted basis independently of
\eqref{eq:type-C-Lyndon-pcentral-conjecture} and identify it with the small Yangian.

\section{Shuffle algebra realization for modular Yangians}\label{sec:modularshufflerealization}


In this section, we shall study the shuffle algebra realization of modular Yangian $Y^{>}_n(\kk)$. First let us define a modular version of shuffle algebra over $\kk$: consider the $\NN^I$-graded $\kk$-vector space 
\[
W^{(n)}(\kk)=\bigoplus_{\underline{k}=(k_1,\dots,k_{n})} W^{(n)}_{\underline{k}}(\kk), 
\]
where $W^{(n)}_{\underline{k}}(\kk)$ consists of $\Sigma_{\underline{k}}$-symmetric rational functions in the variables $\{x_{i,r}\}_{1\leq i\leq n}^{1\leq r\leq k_i}$ satisfying the same pole conditions \eqref{polecondition}  and wheel conditions \eqref{wheel-Y}, but over $\kk$.  With the same shuffle product \eqref{shuffle product-1},  $W^{(n)}(\kk)$ becomes an associative $\kk$-algebra, and the assignment $e_{i,r}\mapsto x^r_{i,1}$ gives rise to a $\kk$-algebra homomorphism
\begin{equation}\label{eq:psi-modular}
 \Psi_{\kk}\colon \ Y^{>}_n(\kk)\longrightarrow W^{(n)}(\kk).
 \end{equation}
However, unlike characteristic $0$ case, $\Psi_\kk$ is not a $\kk$-algebra isomorphism: from Lemma \ref{k times of r-powerof x1}, we have $\Psi_{\kk}(e^p_{i,r})=0$ for any $1\leq i\leq n,\ r\in\NN$. The main result of this section is to determine the kernel $\mathrm{ker}(\Psi_{\kk})$ and  image $\mathrm{im}(\Psi_{\kk})$.

\subsection{The image $\mathrm{im}(\Psi_{\kk})$ -- rank $1$ case}
 We first consider the rank $1$ case. As in  subsection \ref{subs:zbasis}, let
\[
W^{(1)}_{\hbar}(\kk)=\bigoplus_{k\in\NN} W^{(1)}_{\hbar,k}(\kk),\quad \text{in which}\  W^{(1)}_{\hbar,k}(\kk)=\kk[x_1,\dots,x_k]^{\Sigma_k},
\]
and the shuffle product be defined as \eqref{eq:shuffleproducta1hbar}, with $\hbar\in\kk^{\times}$, that is $\hbar\neq 0$. Consider  the subalgebra $\tilde{W}_{\hbar}^{(1)}(\kk)$ generated by $x^{r}_{1}\  (r\in\NN)$.

For a partition $\lambda\in\cP$, 
recall that $m_i(\lambda)$ is the number of parts $\lambda_j$ which are equal to $i$. A partition $\lambda\in\cP$ is called {\em $p$-restricted} if $m_i(\lambda)<p$ for all  $i\in\NN$. For any $k\in\NN$, 
let $I_k\in W^{(1)}_{\hbar,k}(\kk)$ be the $\kk$-subspace spanned by the Hall-Littlewood polynomials $P_{\lambda}$ associated to the $p$-restricted partitions, cf. \eqref{eq:halllittlewoodpolynomial}, then $I_{k}\subset \tilde{W}_{\hbar}^{(1)}(\kk)\cap W^{(1)}_{\hbar,k}(\kk)$. Similar to \cite[(3.1)]{FJMMT03},  we consider another subspace of $W^{(1)}_{\hbar,k}(\kk)$ defined by the following wheel condition:
\begin{equation}\label{eq:wheelp}
 J_k\coloneqq \{f\in W^{(1)}_{\hbar,k}(\kk);~f(x_1,x_2,\dots,x_k)=0\ \text{once}\ x_1=x_2+\hbar=\cdots=x_p+(p-1)\hbar\}.   
\end{equation}

\begin{Proposition}\label{prop:rank1-wheel}
$\tilde{W}^{(1)}_{\hbar}(\kk)\cap W^{(1)}_{\hbar,k}(\kk)\subset J_k$.    
\end{Proposition}
\begin{proof}
If $k<p$, then condition \eqref{eq:wheelp} is vacuous,
and $J_k=W^{(1)}_{\hbar,k}(\kk)$.
If $k\geq p$, then it suffices to show that for any $\mu_1,\dots,\mu_k\in \NN$, 
 \[g(x_1,\dots,x_k)=x^{\mu_1}_{1}\star \cdots \star x^{\mu_k}_1=\sum_{\sigma\in \Sigma_{k}}\left(x^{\mu_1}_{\sigma(1)}x^{\mu_2}_{\sigma(2)}\cdots x^{\mu_k}_{\sigma(k)}\prod_{1\leq i<j\leq k}\tfrac{x_{\sigma(i)}-x_{\sigma(j)}+\hbar}{x_{\sigma(i)}-x_{\sigma(j)}}\right)\in  J_k.\]
 Let $x_1=x_2+\hbar=\cdots+x_p+(p-1)\hbar$. If there is $1\leq \ell\leq p-1$ such that $\sigma^{-1}(\ell)>\sigma^{-1}(\ell+1)$,
 and suppose $\sigma^{-1}(\ell)=j$, 
 $\sigma^{-1}(\ell+1)=i$, then the term corresponding to $\sigma$ in the above symmetrization is zero due to the factor $\frac{x_{\sigma(i)}-x_{\sigma(j)}+\hbar}{x_{\sigma(i)}-x_{\sigma(j)}}$; if $\sigma^{-1}(\ell)<\sigma^{-1}(\ell+1)$ for any $1\leq \ell\leq p-1$, and  suppose $\sigma^{-1}(1)=i$, $\sigma^{-1}(p)=j$, 
 then the term corresponding to $\sigma$ in the above symmetrization is also zero due to the factor $\frac{x_{\sigma(i)}-x_{\sigma(j)}+\hbar}{x_{\sigma(i)}-x_{\sigma(j)}}$. Therefore $g(x_1,\dots,x_k)\in J_k$.
\end{proof}

\begin{Proposition}
 For any $k\in\NN$, we have  $J_k\subset I_k$.  
\end{Proposition}
\begin{proof}
If $k<p$, 
every partition of length $k$ is $p$-restricted, thus from subsection \ref{subs:zbasis} we know $I_k=J_k=\kk[x_1,\dots,x_k]^{\Sigma_k}$.
Next, for any $k\geq p-1$, we show that if $I_k=J_k$,
then $J_{k+1}\subset I_{k+1}$.

For any $f\in J_{k+1}$,
there is a unique decomposition $f=x_1\cdots x_{k+1}\cdot g+h$,
in which $g,h\in \kk[x_1,\dots,x_{k+1}]^{\Sigma_{k+1}}$,
and
 \begin{equation}\label{eq:conditionh}
 h=\sum_{\lambda\in \cP_{k+1}} b_{\lambda} m_{\lambda},\quad \text{where all}\ \lambda\ \text{with}\ b_\lambda\neq 0\ \text{satisfies}\ m_{0}(\lambda)>0. 
 \end{equation}
 Consider the map $\rho\colon \kk[x_1,\dots,x_{k+1}]^{\Sigma_{k+1}} \longrightarrow  \kk[x_1,\dots,x_{k}]^{\Sigma_{k}}$, which sends
  \begin{equation*}
q(x_1,\dots,x_k,x_{k+1})\mapsto q(x_1,\dots,x_k,0),
  \end{equation*}
then we have $\rho(h)=\rho(f)\in J_k=I_k$, that is $\rho(h)=\sum_{\lambda\in \cP_k} c_{\lambda}P_{\lambda}$, where  all $\lambda$ with $c_\lambda\neq 0$ are $p$-restricted and they constitute a  finite set.
Let $<$ stand for the lexicographical order on $\cP_k$:
\[\lambda<\mu \quad \text{if}\quad  |\lambda|<|\mu|\ \text{or}\ |\lambda|=|\mu|,\ \text{and}\ \exists\ j\ \text{s.t.}\ \lambda_i=\mu_i\ \text{for}\ i<j,\ \text{but}\ \lambda_j<\mu_j. \]

\smallskip
  
\textbf{Claim:}~If $\rho(h)\neq 0$, and $\rho(h)=c_{\gamma}P_{\gamma}+\sum_{\mu<\gamma} c_{\mu}P_{\mu}$, with $c_{\gamma}\neq 0$, then  $m_0(\gamma)\leq p-2$.
\smallskip

\noindent To prove the claim, first note that we can also write $\rho(h)$ as a linear combination of monomial basis, so that  $\rho(h)=c_{\gamma} m_{\gamma}+\sum_{\mu<\gamma} a_{\mu}m_{\mu}$.
For any $\lambda=(\lambda_1\geq \cdots\geq \lambda_k)\in\cP_k$, let $\tilde{\lambda}=(\lambda_1\geq \cdots \geq \lambda_k\geq 0)\in\cP_{k+1}$ be the associated partition of length $k+1$, then $\rho(m_{\tilde{\lambda}})=m_{\lambda}$. Since $h$ satisfies the condition \eqref{eq:conditionh}, we see $h=c_{\gamma} m_{\tilde{\gamma}}+\sum_{\mu<\gamma} a_{\mu}m_{\tilde{\mu}}$.
Now consider the map
$\tau\colon \kk[x_1,\dots,x_{k+1}]^{\Sigma_{k+1}} \longrightarrow  \kk[x_1,\dots,x_{k-p+1}]^{\Sigma_{k-p+1}}$,
which sends
\begin{equation*}
q(x_1,\dots,x_{k-p+1},x_{k-p+2},x_{k-p+3},\dots,x_{k+1})\mapsto q(x_1,\dots,x_{k-p+1},0,\hbar,\dots,(p-1)\hbar).
  \end{equation*}
For $\mu=(\mu_1\geq \mu_2\geq  \cdots\geq  \mu_{k+1})\in\cP_{k+1}$ with  $m_0(\mu)>0$, let $\mu'=(\mu_1\geq \mu_2\geq\cdots\geq \mu_{k-p+1} )$ be the associated partition of length $k-p+1$, then we have $\tau(m_{\mu})=\sum_{\lambda\leq \mu'} b_{\lambda} m_{\lambda}$. If $m_{0}(\gamma)>p-2$, then $m_0(\gamma)=p-1$ (since $\gamma$ is $p$-restricted),   $m_{0}(\tilde{\gamma})=p$, thus for any $\mu\in \cP_{k+1}$ with $\mu<\tilde{\gamma}$, we have $\mu'<\tilde{\gamma}'$,  where $\tilde{\gamma}'=(\gamma_1\geq \cdots\geq \gamma_{k-p+1})$ and $\mu'=(\mu_1\geq \cdots\geq \mu_{k-p+1} )$. Then it follows from   $f\in J_{k+1}$ that
\[
0=\tau(f)=\tau(h)=c_{\gamma}\tau(m_{\tilde{\gamma}})+\sum_{\mu<\gamma} a_{\mu}\tau(m_{\tilde{\mu}})=c_{\gamma} m_{\tilde{\gamma}'}+\text{lower terms},   
\]
 which is a contradiction. Hence  $m_{0}(\gamma)\leq p-2$ and the above claim is true. 
 
\smallskip

Therefore, if $\rho(h)=c_{\gamma}P_{\gamma}+\sum_{\mu<\gamma} c_{\mu}P_{\mu}$, with $c_{\gamma}\neq 0$, then the associated partition $\tilde{\gamma}$ is also $p$-restricted.
We can form the Hall-Littlewood polynomial $P_{\tilde{\gamma}}$. From \eqref{eq:plambdamlambda} and $\rho(m_{\tilde{\gamma}})=m_{\gamma}$,  we have $\rho(P_{\tilde{\gamma}})=P_{\gamma}+Q_\gamma$, with $\deg Q_{\gamma}< \deg P_{\gamma}$.
Consider 
\[\tilde{f}=f-c_{\gamma}P_{\tilde{\gamma}},\]
then $\tilde{f}\in J_{k+1}$. Let   $\tilde{f}=x_1\cdots x_{k+1}\cdot \tilde{g}+\tilde{h}$, with $\tilde{g},\tilde{h}\in \kk[x_1,\dots,x_{k+1}]^{\Sigma_{k+1}}$ and $\tilde{h}$ satisfies condition \eqref{eq:conditionh}. Then \[\rho(\tilde{h})=\rho(\tilde{f})=\rho(f)-c_{\gamma} \rho(P_{\tilde{\gamma}})=\sum_{\mu<\gamma}c_{\mu} P_{\mu}-c_{\gamma}Q_{\gamma}.\]
Since $\rho(\tilde{h})=\rho(\tilde{f})\in J_{k}=I_k$, if $\rho(\tilde{h})\neq 0$, we have $\rho(\tilde{h})=b_{\nu} P_{\nu}+\sum_{\mu<\nu} b_{\mu}P_{\mu}$ , with $\nu<\gamma$. Iterating the above procedure, we know for any $f\in J_{k+1}$, there is $w \in I_{k+1}$ with $f-w=x_1\cdots x_{k+1} \cdot g'+h'$, with $\rho(h')=0$. Since $h'$ also  satisfies \eqref{eq:conditionh}, we know $h'=0$ and $g'\in J_{k+1}$. Continuing with induction on the degree of $f$, we finally show $J_{k+1}\subset I_{k+1}$.
\end{proof}

Combining the inclusions $I_{k}\subset \tilde{W}^{(1)}_{\hbar}(\kk)\cap W^{(1)}_{\hbar,k}(\kk)\subset J_k\subset I_k$, we obtain:

\begin{Theorem}\label{thm:rank1basis}
For any $
\hbar\in\kk^{\times}$, let $\tilde{W}^{(1)}_{\hbar}(\kk)$ be the subalgebra of $W^{(1)}_{\hbar}(\kk)$ generated by $x^{r}_1\ (r\in\NN)$, then $\tilde{W}^{(1)}_{\hbar}(\kk)$ is determined by the wheel condition: $f\in \tilde{W}^{(1)}_{\hbar}(\kk)$ if and only if
\begin{equation}\label{eq:wheelcharp}
f(x_1,x_2,\dots,x_k)=0\ \text{once}\ x_1=x_2+\hbar=\cdots=x_p+(p-1)\hbar,   
\end{equation}   
and the set of Hall-Littlewood polynomials $P_{\lambda}$ (cf. \eqref{eq:halllittlewoodpolynomial}) associated to $p$-restricted partitions form a basis of $\tilde{W}_{\hbar}^{(1)}(\kk)$.   
\end{Theorem}

Let $\hbar=1$, then $W^{(1)}_{\hbar=1}(\kk)$ is the  modular shuffle algebra of type $A_1$, and $\Psi_{\kk}(Y^{>}_1(\kk))=\tilde{W}^{(1)}_{\hbar=1}(\kk)$, thus Theorem \ref{thm:rank1basis} determines  $\mathrm{im}(\Psi_{\kk})$ in $A_1$ case.

\subsection{The image $\mathrm{im}(\Psi_{\kk})$ -- general case} \label{subs:impsikk}

Let us now describe the image $\mathrm{im}(\Psi_{\kk})$ in general.  In characteristic $p>3$ case, we still have the  specialization maps
\begin{equation*}
  \phi_{\underline{d}}\colon \ W^{(n)}_{\underline{k}}(\kk)\longrightarrow
  \kk[\{w_{\beta,s}\}_{\beta\in\Delta^{+}}^{1\leq s\leq d_{\beta}}]^{\Sigma_{\unl{d}}},
\end{equation*}
defined as \eqref{spe-A}--\eqref{spe-D-2}. Lemmas \ref{lem:spepsipbwbasis}--\ref{lem:spefiltration} still hold (with coefficients in $\kk$,  root vectors $e_{\beta}(r)$ and ordered monomials $E_h$   defined in the same way as over $\CC$), since their proofs rely only on analyzing the specializations of $\zeta$-factors. Moreover, pick any $F\in W^{(n)}_{\unl{k}}(\kk)$ and $\unl{d}\in\mathrm{KP}(\unl{k})$, the specialization
$\phi_{\unl{d}}(F)$ is divisible by the product $\prod_{\beta\in\Delta^{+}}G_{\beta}$,  with $G_{\beta}$ appeared in \eqref{eq:spephidEh} (the proof is the same as over $\CC$ since it is based solely on the wheel conditions).  We can now define the following {\em reduced specialization map}
\begin{equation}
  \xi_{\unl{d}}\colon\  W^{(n)}_{\unl{k}}(\kk) \longrightarrow
  \kk[\{w_{\beta,s}\}_{\beta\in\Delta^{+}}^{1\leq s\leq d_{\beta}}]^{\Sigma_{\unl{d}}}
  \qquad \text{via} \qquad \xi_{\unl{d}}(F)\coloneqq \frac{\phi_{\unl{d}}(F)}{\prod_{\beta\in\Delta^{+}}G_{\beta}}.
\label{reducedspe}
\end{equation}
Let $\tilde{W}^{(n)}_{\unl{k}}(\kk)$ be the subspace  of $W^{(n)}_{\unl{k}}(\kk)$ consisting of functions $F$ that  satisfies : 
\begin{itemize}[leftmargin=0.7cm]
    \item Let $G=\xi_{\unl{d}}(F)$,  that is the symmetric polynomial in $\{w_{\beta,s}\}_{\beta\in\Delta^{+}}^{1\leq s\leq d_{\beta}}$ obtained by reduced specialization map, then $G(\{w_{\beta,s}\}_{\beta\in\Delta^{+}}^{1\leq s\leq d_{\beta}})=0$ once there is $\beta\in\Delta^{+}$ such that
    \begin{equation}\label{eq:wheel-higherrank}
    w_{\beta,1}=w_{\beta,2}+\tfrac{(\beta,\beta)}{2}=\cdots=w_{\beta,p}+(p-1)\cdot \tfrac{(\beta,\beta)}{2}.
\end{equation}
\end{itemize}
Let $\tilde{W}^{(n)}(\kk)=\bigoplus_{\unl{k}\in\NN^{n}} \tilde{W}^{(n)}_{\unl{k}}(\kk)$. Similar to the proofs of \cite[Lemma 3.51]{Tsy21}\cite[Propositions 4.11]{HT24}\cite[Propositions 3.13, 4.11]{HT26}, we have

\begin{Proposition}
 $\mathrm{im}(\Psi_{\kk})\subset \tilde{W}^{(n)}(\kk)$. 
\end{Proposition}

For any $h \colon  \Delta^{+}\times\NN\rightarrow \NN$ with finite support, we say $h$ is {\em $p$-restricted} if $ h(\beta,r)<p$ for all $(\beta,r)\in\Delta^{+}\times \NN$. Denote the set of $p$-restricted functions $h$ by $\mathcal{H}^{p}$, and the set of functions in  $\mathcal{H}^{p}$ with grading $\unl{k}$ by $\mathcal{H}^{p}_{\unl{k}}$.  
We consider the $p$-restricted ordered monomials
\begin{equation}\label{eq:prestrictedordered} 
 E_{h}\, =\prod_{(\beta,r)\in\Delta^{+}\times\mathbb{N}}\limits^{\rightarrow}e_{\beta}(r)^{h(\beta,r)},\qquad h\in\mathcal{H}^{p}
 \end{equation}
as \eqref{eq:pbwd-yangian}, for the Yangian $Y^>_n(\kk)$.
Based on Theorem \ref{thm:rank1basis}, and the analogue of Lemmas \ref{lem:spepsipbwbasis}--\ref{lem:spefiltration}, the following  analogue of Lemma \ref{prop:mainpropertyspe} holds (arguing in the same way as \cite[Section 3.2]{Tsy21}):
\begin{Lemma}\label{lem:charpspan}
 Let $W'_{\unl{k}}(\kk)$ be the subspace of $W^{(n)}_{\unl{k}}(\kk)$ spanned by $\{\Psi_{\kk}(E_{h})\}_{h\in \mathcal{H}^{p}_{\unl{k}}}$. For any $F\in \tilde{W}^{(n)}_{\unl{k}}(\kk)$ and $\unl{d}\in \text{\rm KP}(\unl{k})$, if $\phi_{\unl{d}'}(F)=0$ for all
$\unl{d}'\in \text{\rm KP}(\unl{k})$ such that $\unl{d}'<\unl{d}$, then there exists $F_{\unl{d}}\in W'_{\unl{k}}(\kk)$
such that $\phi_{\unl{d}}(F)=\phi_{\unl{d}}(F_{\unl{d}})$ and $\phi_{\unl{d}'}(F_{\unl{d}})=0$ for all $\unl{d}'<\unl{d}$.
\end{Lemma}

Thus we obtain $\mathrm{im}(\Psi_{\kk})=\tilde{W}^{(n)}(\kk)$, and:

\begin{Proposition}\label{prop:impsikgeneral}
The $\kk$-algebra homomorphism $\Psi_{\kk}$ of \eqref{eq:psi-modular} gives rise to a surjective map 
\begin{equation}\label{eq:surjectivemap}
 \Psi_{\kk}\colon \ Y^{>}_n(\kk)\twoheadrightarrow \tilde{W}^{(n)}(\kk).
 \end{equation}
 Moreover, $\{\Psi_{\kk}(E_h)\}_{h\in\mathcal{H}^{p}}$ form a basis of $\tilde{W}^{(n)}(\kk)$.
\end{Proposition}

\subsection{The small Yangian $\bar{y}^{>}_n(\kk)$}\label{subsection: the subalgebra}
To determine the kernel $\mathrm{ker}(\Psi_{\kk})$, we need another version of modular Yangian $\bar{Y}^{>}_n(\kk)$  constructed    from the integral form $\Y^{>}_n$ and ``reduction modulo $p$''. 
Since $2$ is invertible in $\kk$, there is a natural ring homomorphism $\eta\colon \ZZ[\frac{1}{2}]\rightarrow \kk$, which gives rise to the $\kk$-algebra  $\bar{Y}^{>}_n(\kk)\coloneqq \Y^{>}_n\otimes_{\ZZ[\frac{1}{2}]} \kk$. For any $a\in \Y^{>}_n$, we use $\bar{a}\coloneqq a\otimes 1$ to denote the corresponding element in $\bar{Y}^{>}_n(\kk)$. Note that $\bar{Y}^{>}_n(\kk)$ is not isomorphic to $Y^{>}_n(\kk)$, as $\bar{e}_{\beta}(r)^p=0$ (cf. Proposition \ref{prop:dividedpoweresebetar}) in $\bar{Y}^{>}_n(\kk)$ for any $(\beta,r)\in \Delta^{+}\times \NN$. 

Recall the integral form $\W^{(n)}$ of shuffle algebra from Subsection \ref{subsection:integralformddd},  analogously we can form the $\kk$-algebra $ \W^{(n)}\otimes_{\ZZ[\frac{1}{2}]} \kk$.  The $\ZZ[\frac{1}{2}]$-algebra isomorphism $\Psi\colon \Y_n^> \stackrel{\sim}{\longrightarrow} \W^{(n)}$ (cf. Theorem \ref{ShuffleIY-A}) gives rise to a $\kk$-algebra isomorphism  
\[\bar{\Psi} \colon \bar{Y}_n^>(\kk) \stackrel{\sim}{\longrightarrow} \W^{(n)}\otimes_{\ZZ[\frac{1}{2}]}\kk.\]
 Analogously to the Lie algebra case (cf. \cite{Hum77}), consider the subalgebra $\bar{y}^{>}_n(\kk)$ of $\bar{Y}^{>}_n(\kk)$ generated by  the elements
\[\{\bar{e}_{\beta}(r)^{(t)};\   \beta\in\Delta^{+}, r\in\NN, 0\leq t\leq p-1\},\]
cf. \eqref{def:divided power:positive root},
and we will call $\bar{y}^{>}_n(\kk)$ the \emph{small Yangian}.
We know $\bar{y}^{>}_n(\kk)$ is also generated by $\bar{e}_{i,r}\ (1\leq i\leq n, r\in\NN)$. Thus the $\kk$-algebra isomorphism  $\bar{\Psi} \colon \bar{Y}_n^>(\kk) \stackrel{\sim}{\longrightarrow} \W^{(n)}\otimes_{\ZZ[\frac{1}{2}]}\kk$ gives rise to a $\kk$-algebra isomorphism   \begin{equation}\label{eq:shufflely}
 \bar{\Psi} \colon \bar{y}_n^>(\kk) \stackrel{\sim}{\longrightarrow} \bar{W}^{(n)}(\kk),   
\end{equation}
where $\bar{W}^{(n)}(\kk)$ is the $\kk$-subalgebra of  $W^{(n)}\otimes_{\ZZ[\frac{1}{2}]}\kk$ generated by $x^r_{i,1}\otimes 1\ (1\leq i\leq n,r\in\NN)$.

\begin{Proposition}\label{prop:pbwsmallyangian}
(1) The assignment $f\otimes \lambda\mapsto \lambda\cdot \eta(f)$ gives rise to a $\kk$-algebra isomorphism 
 \begin{equation}\label{eq:isoshuffleoverkreduction}
 \phi\colon\ \bar{W}^{(n)}(\kk)\stackrel{\sim}{\longrightarrow} \tilde{W}^{(n)}(\kk),
   \end{equation}
where $\eta(f)\in \tilde{W}^{(n)}(\kk)$ is given by changing the coefficients from  $\ZZ[\frac{1}{2}]$ to $\kk$ with the map $\eta$.   

\smallskip

\noindent (2) Let $\bar{E}_h\coloneqq E_h\otimes 1$ be the corresponding ordered monomials in $\bar{Y}^{>}_n(\kk)$, cf. \eqref{eq:pbwd-yangian}. Then the elements $\{\bar{E}_{h}\}_{h\in\mathcal{H}^{p}}$  form a basis of $\bar{y}^>_n$ over $\kk$.
\end{Proposition}
\begin{proof}
On the one hand,  $\bar{y}^{>}_n(\kk)$ is generated by $\bar{e}_{i,r}\ (1\leq i\leq n, r\in\NN)$, thus the assignment $e_{i,r}\mapsto \bar{e}_{i,r}$ gives a surjective $\kk$-algebra homomorphism
\begin{equation}\label{eq:mapvarphi}
   \varphi\colon Y^{>}_n(\kk)\twoheadrightarrow \bar{y}^{>}_n(\kk). 
\end{equation}
From Theorem \ref{thm:yangian-basisoverk} and that $\varphi(e_{\beta}(r)^{p})=\bar{e}_{\beta}(r)^{p}=0$, we know $\{\bar{E}_{h}\}_{h\in\mathcal{H}^{p}}$  form a spanning set of $\bar{y}^>_n$ over $\kk$. On the other hand, from subsection \ref{subs:impsikk}, $\tilde{W}^{(n)}(\kk)$ is the subalgebra of $W^{(n)}(\kk)$ generated by $x^{r}_{i,1}\ (1\leq i\leq n, r\in\NN)$ over $\kk$.  Since shuffle product $\star$ is compatible with the map $\eta\colon \ZZ[\frac{1}{2}]\rightarrow \kk$, and $\phi(x^r_{i,1}\otimes 1)=x^r_{i,1}$,  $\phi$ is also a surjective $\kk$-algebra homomorphism. Thus we obtain a surjective map $\phi\circ \bar{\Psi}\colon \bar{y}^>_n \twoheadrightarrow \tilde{W}^{(n)}(\kk)$. Now it follows from 
\[\phi(\bar{\Psi}(\bar{E}_h))=\Psi_{\kk}(E_h)\]
 and Proposition \ref{prop:impsikgeneral} that $\{\bar{E}_{h}\}_{h\in\mathcal{H}^{p}}$ is also linearly independent in $\bar{y}^>_n $. Hence $\{\bar{E}_{h}\}_{h\in\mathcal{H}^{p}}$  form a basis of $\bar{y}^>_n$ over $\kk$,  and $\phi$ is an isomorphism.
\end{proof}

 Combining Propositions \ref{prop:impsikgeneral}--\ref{prop:pbwsmallyangian} and the isomorphism \eqref{eq:shufflely}, we obtain the following commutative diagram
\[
\begin{tikzcd}
Y^{>}_n(\kk) \arrow[d, "\varphi"', two heads] \arrow[r, "\Psi_{\kk}", two heads] & \tilde{W}^{(n)}(\kk)  \\
\bar{y}^{>}_n(\kk) \arrow[r, "\bar{\Psi}"', "\cong"]                             & \bar{W}^{(n)}(\kk) \arrow[u, "\phi", "\cong"']
\end{tikzcd},
\]
such that $\Psi_{\kk}=\phi\circ \bar{\Psi}\circ \varphi$. Recall the ideal $I_p(Y_n^>(\kk))$ and the quotient $Y_n^{>,[p]}$ from \eqref{eq:defresyangian}.  Since $\varphi(e_\beta(r)^p)=\bar e_\beta(r)^p=0$, we have $I_p(Y_n^>(\kk))\subset\ker(\Psi_\kk)$.  Every non-$p$-restricted PBW monomial contains a consecutive factor $e_\beta(r)^p$, so the classes
\[\hat{E}_{h}\coloneqq \prod_{(\beta,r)\in\Delta^{+}\times\mathbb{N}}\limits^{\rightarrow}\hat{e}_{\beta}(r)^{h(\beta,r)},\qquad h\in\mathcal{H}^{p}\]
span $Y_n^{>,[p]}$, where $\hat{e}_{\beta}(r)$ is the class of $e_{\beta}(r)$.  Proposition~\ref{prop:impsikgeneral} says that their images $\Psi_\kk(E_h)$ form a basis of $\tilde W^{(n)}(\kk)$.  Hence the classes $\hat E_h$ are linearly independent as well.  Thus they form a basis of $Y_n^{>,[p]}$ and $\ker(\Psi_{\kk})=I_p(Y_n^{>}(\kk))$.  This argument does not use $p$-centrality and therefore also applies in type $C$.  We obtain our main result:

\begin{Theorem}\label{maintheorem}
The $\kk$-algebra homomorphism 
 $\Psi_{\kk}\colon Y^{>}_n(\kk)\longrightarrow W^{(n)}(\kk)$ of \eqref{eq:psi-modular} gives rise to  a $\kk$-algebra isomorphism $\bar{\Psi}_{\kk}\colon Y^{>,[p]}_n \stackrel{\sim}{\longrightarrow} \tilde{W}^{(n)}(\kk)$, and we have the following commutative diagram of $\kk$-algebra isomorphisms:
 \[
\begin{tikzcd}
Y^{>,[p]}_n \arrow[d, "\bar{\varphi}", "\cong"'] \arrow[r, "\bar{\Psi}_{\kk}", "\cong"'] & \tilde{W}^{(n)}(\kk)  \\
\bar{y}^{>}_n(\kk)  \arrow[r, "\bar{\Psi}"', "\cong"]                            & \bar{W}^{(n)}(\kk) \arrow[u, "\phi", "\cong"']
\end{tikzcd}
\]
\end{Theorem}


\bigskip
\noindent
\textbf{Acknowledgment.}
We would like to thank Boris Feigin and Alexander Tsymbaliuk for their inspiring  work and for stimulating discussions. This work is supported by the National Natural 
Science Foundation of China (Nos. 12401028, 12671035) and
the Natural Science Foundation of Hubei Province (No. 2025AFB716).


\end{document}